\documentclass{article}
\usepackage{amsmath, amsthm, amssymb, amsfonts, enumitem, comment, lineno, verbatim}
\usepackage[margin=3cm]{geometry}
\usepackage[capitalize]{cleveref}

\newcommand{\N}{\mathbb{N}}

\newcommand{\cF}{\mathcal{F}}

\newcommand{\cP}{\mathcal{P}}

\newcommand{\aut}{\mathrm{Aut}}
\newcommand{\se}{\subseteq}
\newcommand{\sm}{\setminus}

\newcommand{\id}{\mathrm{id}}

\newcommand{\per}{\mathrm{per}}

\newtheorem{thm}{Theorem}
\numberwithin{thm}{section}
\newtheorem{lemma}[thm]{Lemma}
\newtheorem{prop}[thm]{Proposition}
\newtheorem{cor}[thm]{Corollary}
\newtheorem*{conj}{Conjecture}

\theoremstyle{definition}
\newtheorem{defn}{Definition}
\numberwithin{defn}{section}
\newtheorem{alg}{Algorithm}
\numberwithin{alg}{section}

\theoremstyle{remark}
\newtheorem{rmk}{Remark}
\numberwithin{rmk}{section}

\begin{document}

\title{The road problem and homomorphisms of directed graphs}
\author{Sophie MacDonald}

\maketitle

\begin{abstract}
We make progress on a generalization of the road (colouring) problem. The road problem was posed by Adler-Goodwyn-Weiss and solved by Trahtman. The generalization was posed, and solved in certain special cases, by Ashley-Marcus-Tuncel. We resolve two new families of cases, of which one generalizes the road problem and follows Trahtman's solution, and the other generalizes a result of Ashley-Marcus-Tuncel with a proof quite different from theirs. Along the way, we prove a universal property for the fiber product of certain graph homomorphisms, which may be of independent interest. We provide polynomial-time algorithms for relevant constructions and decision problems.
\end{abstract}


\section{Introduction}\label{sec-intro}

\subsection{Background: from toral automorphisms to finite automata}

In this paper, we are concerned with a generalization of the road (colouring) problem. That problem, posed by Adler-Goodwyn-Weiss in \cite{agw-77-tms} and solved by Trahtman in \cite{ant-09-ijm}, asked whether every strongly connected, aperiodic directed graph of constant out-degree is the underlying graph of a synchronizing deterministic finite automaton (DFA). Trahtman's road colouring theorem (\cref{thm-rct-aper} below) gives an affirmative answer, and \cite{bf-11-aaecc,bp-14-dam} give a generalization to periodic graphs (\cref{thm-prct}).

The motivation for the road problem comes from ergodic theory. Specifically, a weak form of the road colouring theorem was used to prove the main theorem of \cite{agw-77-tms}, which gives a criterion for measure-theoretic isomorphism of certain Markov chains. This was the first time the road problem was explicitly posed, although the real origin of the problem is the earlier paper \cite{aw-67-pnas}, which concerns the analogous isomorphism problem for hyperbolic automorphisms of the two-dimensional torus. 

The present paper concerns a generalization of the road problem motivated by graph-theoretic invariants for a different, but related, isomorphism relation in ergodic theory. Specifically, Ashley-Marcus-Tuncel \cite{amt-95-etds} identified a graph-theoretic criterion for isomorphism of one-sided stationary Markov chains, implicit in \cite{bt-90-etds}, and gave a complete, effectively computable set of isomorphism invariants. They observed that a certain conjectural uniqueness result (the $O(G)$ conjecture, below in \S \ref{subsec-trans-stab-sync}) would, if proven, simplify the set of invariants, and proved the conjecture in a family of special cases. The conjecture reduces to a generalization of the road problem involving certain right-resolving graph homomorphisms (which we call \textit{right-resolvers}; see \S \ref{subsec-graph-homo-basic-defns}). For graphs of constant out-degree, these homomorphisms coincide with road colourings.

\subsection{Main ideas and contributions}

The main purpose of this paper is to present new results toward the $O(G)$ conjecture. There are two main ideas in the paper. The first idea concerns the \textit{stability relation} of a right-resolver, which was introduced in \cite{ckk-02-lncs} for DFAs or road colourings. For context, Kari \cite{jk-03-tcs} solved the road problem in the Eulerian case by finding, for a given graph, a road colouring with nontrivial (i.e. not merely diagonal) stability relation, then recursively finding a synchronizing road colouring of the strictly smaller quotient graph (in which the states are stability classes), and lifting it to the original graph. Trahtman's solution of the full road problem uses the same inductive strategy, paired with a more sophisticated technique for obtaining a colouring with a nontrivial stability relation. 

By determining how the stability relation behaves with respect to composition of right-resolvers, we are able to study the aforementioned recursive lifting constructions more systematically, allowing us to apply them toward the $O(G)$ conjecture. In particular, we adapt Trahtman's proof of the road colouring theorem to cover a larger family of cases (\cref{thm-bfc-cyc-bunch}), in which the out-degrees of states are allowed to vary cyclically. To do so, we generalize, to the setting of right-resolvers, a sufficient condition for a road colouring to have a nontrivial stable pair, based on the idea of a function graph with a unique tallest tree. This condition is at the heart of all proofs of the road colouring theorem to date, and its importance has motivated detailed analysis \cite{ber-16-lncs}.

The second main idea in this paper is a graph property that we call \textit{bunchiness}, along with a weaker property called \textit{almost bunchiness}. Bunchy and almost bunchy graphs are characterized by the property that the right-resolvers they admit are unique up to automorphisms in a certain sense (\cref{prop-alm-bunchy-equiv}). We highlight the implicit role of bunchiness both in \cite{amt-95-etds} and in the road colouring literature, and prove the $O(G)$ conjecture for bunchy and almost bunchy graphs (\cref{thm-alm-og-exists}). Furthermore, we show that the fiber product of right-resolving homomorphisms satisfies a universal property (\cref{thm-univ-prop}) that further highlights the essential role of bunchy graphs.  

Motivated by these results, we introduce a new conjecture, which we call the bunchy factor conjecture (see \S \ref{sec-conjec}), asserting essentially that the $O(G)$ conjecture can be proved using the stability approach that Trahtman used to prove the road colouring theorem, with bunchy graphs as the base of the recursion. Another way of articulating this conjecture is that the barrier to proving the $O(G)$ conjecture is our lack of a sufficiently general method of producing homomorphisms with nontrivial stability relation. The fact that the bunchy factor conjecture implies the $O(G)$ conjecture is made explicit in \cref{prop-bunchy-implies-o}, which relies primarily on \cref{prop-bunchy-same-o}, a uniqueness result that uses the universal property of the fiber product in an essential way.

\subsection{Organization of the paper}

In \S \ref{sec-graph-homo}, we recall and adapt standard material on graphs and homomorphisms. In \S \ref{sec-stab-sync}, we define the stability relation for a right-resolver, give its main structural properties, and relate it to synchronization. In \S \ref{sec-og-road}, we recall from \cite{amt-95-etds} the connection between the road problem and the $O(G)$ conjecture, and state our generalization of the road colouring theorem. 

In \S \ref{sec-bunchy}, we introduce the concepts of bunchiness and almost bunchiness and present results involving them, including the $O(G)$ conjecture for bunchy and almost bunchy graphs and the universal property of the fiber product. In \S \ref{sec-conjec}, we pose the bunchy factor conjecture, which has several equivalent formulations, and discuss its relation to the $O(G)$ conjecture and the road problem. In \S \ref{sec-comput}, we give polynomial-time algorithms for construction and decision problems involving right-resolvers, and discuss the algorithmic implications of the $O(G)$ and bunchy factor conjectures.

The proofs of many results in \S \S \ref{sec-graph-homo}--\ref{sec-comput}, comprising the structural properties of graphs and right-resolvers, are deferred to \S \ref{sec-proofs}. The proof of our generalization of the road colouring theorem is deferred to \S \ref{sec-traht}.


\section{Graphs and graph homomorphisms}\label{sec-graph-homo}

\subsection{Basic definitions}\label{subsec-graph-homo-basic-defns}

We take all graphs to be finite and directed. A graph $G$ consists of a set $V(G)$ of states, or vertices, and a set $E(G)$ of edges, together with a pair of maps $s, t: E(G) \to V(G)$ giving the source and target of each edge. Loops (edges $e$ with $s(e) = t(e)$) and parallel edges (distinct edges $e,e'$ with $s(e) = s(e')$, $t(e) = t(e')$) are allowed. For $I \in V(G)$, we write $E_I(G) = s^{-1}(I)$ for the set of outgoing edges from $I$. We write $F(I) = t(E_I(G))$ for the set of follower states of $I$, and we write $E_{IJ}(G) = s^{-1}(I) \cap t^{-1}(J)$ for the set of edges from $I$ to $J$. We write $L(G)$ for the language of $G$, i.e. the set of finite edge paths in $G$, i.e. $e_1 e_2 \dots e_n$ where $t(e_i) = s(e_{i+1})$, $1 \leq i \leq n-1$. We also refer to elements of $L(G)$ as \textit{words}. The maps $s,t$ extend to $s,t: L(G) \to V(G)$ by $s(e_1 \dots e_n) = s(e_1)$, $t(e_1 \dots e_n) = t(e_n)$. We define $L_I(G) = \{ u \in L(G) \, | \, s(u) = I  \}$, and $L_{IJ}(G) = \{ u \in L_I(G) \, | \, t(u) = J \}$. A \textit{cycle} in $G$ is a path $u \in L(G)$ with $s(u) = t(u)$, i.e. an element of $L_{II}(G)$ for some $I \in V(G)$.

A \textit{graph homomorphism} $\Phi: G \to H$ is a pair of maps $\Phi: E(G) \to E(H)$, $\partial \Phi: V(G) \to V(H)$ such that $s \circ \Phi = \partial \Phi \circ s$ and $t \circ \Phi = \partial \Phi \circ t$. If there is a surjective homomorphism from $G$ to $H$, then we say that $H$ is a \textit{factor} of $G$, and that $G$ is an \textit{extension} of $H$. Observe that every factor of a strongly connected graph is strongly connected. A graph homomorphism $\Phi: G \to H$ induces a map $L(G) \to L(H)$ (also written $\Phi$) in the obvious way.  

A graph isomorphism is a homomorphism that is injective and surjective (i.e. on both edges and states), and an automorphism of a graph $G$ is an isomorphism from $G$ to itself. We denote the group of automorphisms of $G$ by $\aut(G)$, and by $P(G)$ the (normal) subgroup of $\aut(G)$ which acts trivially on states and permutes parallel edges. We generally identify isomorphic graphs, and use the symbol $=$ to denote isomorphism, except when discussing algorithms for deciding isomorphism, or confirming that the automorphism group of a given graph is trivial.

We now introduce the class of homomorphisms with which we are concerned.

\begin{defn}[right-resolver]\label{defn-rr}
Let $G,H$ be graphs. Let $\Phi: G \to H$ be a surjective graph homomorphism. We say that $\Phi$ is \textit{right-resolving}, or is a \textit{right-resolver}, if, for each $I \in V(G)$, the restriction $\Phi|_{E_I(G)} : E_I(G) \to E_{\partial \Phi(I)} (H)$ is a bijection. We denote the set of right-resolvers $G \to H$ by $\hom_R(G,H)$, and we write $H \leq_R G$ if $\hom_R(G,H) \neq \emptyset$. 
\end{defn}

\begin{rmk}
The term ``right-resolving'' comes from symbolic dynamics, where the words are of primary importance and the actual graph is secondary. A graph homomorphism $\Phi: G \to H$ is right-resolving if and only if the associated map $\Phi: L(G) \to L(H)$ satisfies a certain condition on the symbols (edges) appearing to the right of a given symbol in a word. See \cite{amt-95-etds}, \S 8.2 for details if interested.
\end{rmk}

\begin{rmk}
The class of right-resolving graph homomorphisms is closed under composition. This reduces to the fact, applied to the outgoing edges from each state, that a composition of bijections is a  bijection. This means that the relation $\leq_R$ is transitive, and since the graphs are finite, it is clearly antisymmetric, so it is indeed a partial order on the set of all graphs (really on the set of equivalence classes of graphs up to isomorphism).
\end{rmk}

The following lemma is evident but we state it explicitly for future reference.

\begin{lemma}\label{lemma-fiber-determ-img}
The image of a right-resolver is determined up to graph isomorphism by the partition of the domain into fibers. That is, if $H_1, H_2 \leq_R G$ via $\Phi_i \in \hom_R(G,H_i)$, and for any $I_1, I_2 \in V(G)$ we have $\partial \Phi_1(I_1) = \partial \Phi_1(I_2)$ if and only if $\partial \Phi_2(I_1) = \partial \Phi_2(I_2)$, then in fact the $H_i$ are isomorphic.
\end{lemma}

Note that the converse is not true: for a given $G,H$ with $H \leq_R G$, there may exist $\Phi, \Phi' \in \hom_R(G,H)$ with distinct partitions $\{ (\partial \Phi)^{-1}(I) \, | \, I \in V(H)  \}$, $\{ (\partial \Phi')^{-1}(I) \, | \, I \in V(H)  \}$ of $V(G)$. However, this cannot occur when $H$ is $\leq_R$-minimal:

\begin{thm}[\cite{amt-95-etds}, Theorem 3.2 and Corollary 3.3(a)]\label{thm-mg-sigma}
For any graph $G$, there exist a unique $\leq_R$-minimal graph $M(G) \leq_R G$ and a unique map $\Sigma_G : V(G) \to V(M(G))$ such that $\partial \Phi = \Sigma_G$ for any $\Phi \in \hom_R(G,M(G))$.
\end{thm}

The construction of $M(G)$ was first given in \cite{cg-70-acm}, though not in this notation. We discuss the proof of \cref{thm-mg-sigma} in \S \ref{subsec-right-air-mg}. The notion of $M(G)$, for a graph $G$, provides context for road colourings:

\begin{defn}[$M_D$ and road colourings]
For $D \geq 1$, let $M_D$ be the graph with a single state and $D$ self-loops. For a graph $G$ of constant out-degree $D$, a \textit{road colouring} of $G$ is a right-resolver $G \to M_D$.
\end{defn}

Note that each $M_D$ is $\leq_R$-minimal, and that $\hom_R(G,M_D)$ is nonempty if and only if $G$ has constant out-degree $D$, in which case $M(G) = M_D$.


\subsection{Subgraphs and connectedness}

A \textit{sink} in a graph $G$ is a state $I \in V(G)$ is a state with no outgoing edges, i.e. $F(I) = \emptyset$. We assume throughout that all graphs are sink-free; this is purely for convenience, as all of the results that do not require strong connectedness can be proved for graphs with sinks, with routine but tedious modifications to the proofs.  We say that a graph $G$ is \textit{strongly connected}, or irreducible, if for any ordered pair $I, J \in V(G)$, there is a (directed) edge path in $G$ from $I$ to $J$, i.e. $L_{IJ}(G) \neq \emptyset$. Note that strongly connected graphs are sink-free. The \textit{period} $\per(G)$ of a strongly connected graph $G$ is the gcd of its cycle lengths. 

A graph $H$ is a \textit{subgraph} of a graph $G$ if $E(H) \se E(G)$, $V(H) \se V(G)$, and the maps $s,t$ with respect to $H$ agree with their counterparts on $G$, restricted to $H$. An \textit{induced subgraph} of a graph $G$ is a subgraph $H$ such that $E_{IJ}(H) = E_{IJ}(G)$ for every $I,J \in V(H)$. A \textit{strong component} of a graph is a maximal strongly connected subgraph, i.e. a strongly connected subgraph that is not a proper subgraph of another strongly connected graph.

A \textit{principal subgraph} of a graph $G$ is a subgraph $H$ such that $E_I(H) = E_I(G)$ for every $I \in V(H)$. Note that every principal subgraph is induced. Note also that if $H$ is a principal subgraph of $G$, and $K$ is a principal subgraph of $H$, then $K$ is a principal subgraph of $G$. A \textit{principal component} is a strongly connected principal subgraph; note that the principal components are precisely the minimal principal subgraphs. In particular, any two principal components of a given graph have disjoint sets of states. The principal components of $G$ correspond to the sink states in the \textit{condensation} of $G$, which is the directed acyclic graph in which the states are the strong components, or maximal strongly connected subgraphs, of $G$, and with an edge $C_1 \to C_2$ in the condensation if there is an edge $I_1 \to I_2$ for $I_i \in V(C_i)$ in $G$.

Let $G, H$ be graphs with $H \leq_R G$, and let $\Phi \in \hom_R(G,H)$. Let $K$ be a subgraph of $G$. Note, by the right-resolving property, that in order for $\Phi|_K: K \to H$ to be surjective, it is necessary and sufficient that $\partial \Phi|_{V(K)}: V(K) \to V(H)$ be surjective and that $K$ be a principal subgraph of $G$.

\begin{rmk}
The road problem and the $O(G)$ problem were both originally raised for strongly connected graphs, which is a natural restriction given the origins of both problems in the ergodic theory of stationary Markov chains. Moreover, strong connectedness is used in an important way in a lemma used to prove both the road colouring theorem and the almost bunchy case of the bunchy factor conjecture. This is why the $O(G)$ conjecture and the bunchy factor conjecture are stated only for strongly connected graphs. 

However, it is quite natural from an automata-theoretic perspective, especially concerning computational complexity, to consider graphs that are not strongly connected. For instance, Eppstein \cite{de-90-siam} shows that it is NP-complete to determine whether the minimum length of a synchronizing word for a given synchronizing DFA is at most some given value. Eppstein's examples are not strongly connected. There are also several graph problems from symbolic dynamics \cite{fc-21-arxiv} that are NP-complete in general but have polynomial-time algorithms in the strongly connected case.
\end{rmk}


\section{Stability and synchronization}\label{sec-stab-sync}

\subsection{Transitions, stability, and synchronization}\label{subsec-trans-stab-sync}

A right-resolver on a graph $G$ induces transition maps on $V(G)$ in the standard way:

\begin{defn}[transition map]
Let $G, H$ be graphs with $H \leq_R G$. Let $\Phi \in \hom_R(G,H)$. For $I \in V(G)$ and $u \in L_{\partial \Phi(I)}(H)$, we write $I \cdot u$ for the terminal state $t(\gamma)$ of the unique $\gamma \in L_I(G)$ with $\Phi(\gamma) = u$. That is, $I \cdot u = t( (\Phi|_{L_I(G)} )^{-1}(u))$. We denote by $S_{\Phi}$ the set of maps of the form $I \mapsto I \cdot u$ with respect to $\Phi$.
\end{defn}

We now introduce the notion of a \textit{congruence} (see \cite{bpr-09-ca}, Chapter 1, or \cite{jk-03-tcs}, \S 3), of which we will see two important examples. The main example will be the stability relation, but we will also use a congruence in \cref{prop-construct-bg} to construct the maximal bunchy factor $B(G)$ of a given graph $G$.

\begin{defn}[congruences and quotients]
Let $G,H$ be graphs with $H \leq_R G$, let $\Phi \in \hom_R(G,H)$, and let $\sim$ be an equivalence relation on $V(G)$. We say that $\sim$ is a \textit{congruence} with respect to $\Phi$ if it is invariant under transitions, i.e. for all $I \in V(H)$, all $u \in L_I(H)$, and all $I_1, I_2 \in (\partial \Phi)^{-1}(I)$ with $I_1 \sim I_2$, we have $I_1 \cdot u \sim I_2 \cdot u$. We ``overload'' a congruence $\sim$ by defining it also on paths (in particular, edges), by saying that $\gamma_1 \sim \gamma_2$, for $\gamma_1, \gamma_2 \in L(G)$, if $\Phi(\gamma_1) = \Phi(\gamma_2)$ and $s(\gamma_1) \sim s(\gamma_2)$. Define the \textit{quotient graph} $G/\sim$ by $V(G/\sim) = V(G)/\sim$, $E(G/\sim) = E(G)/\sim$, $s([e]_{\sim}) = [s(e)]_{\sim}$, and $t([e]_{\sim}) = [t(e)]_{\sim}$. 
\end{defn}

\begin{rmk}
Let $G,H$ be graphs with $H \leq_R G$, let $\Phi \in \hom_R(G,H)$, and let $\sim$ be a congruence on $G$ with respect to $\Phi$. Observe that there are right-resolvers $G \to G/\sim$, $G/\sim \, \to H$ which compose to $\Phi$, where the right-resolver $G \to G/\sim$ is the quotient map, and the right-resolver $G/\sim \, \to H$ takes a $\sim$ class to the image in $H$ of any of its representatives.
\end{rmk}

\begin{rmk}
The coarsest congruence, with respect to a right-resolver $\Phi \in \hom_R(G,H)$, is the total relation on the fibers, i.e. the relation $\bigsqcup_{I \in V(H)} ((\partial \Phi)^{-1}(I))^2$. The quotient of $G$ by this relation is simply $H$, with $\Phi$ as the quotient map.
\end{rmk}

\begin{defn}[stability relation for a right-resolver]
Let $G,H$ be graphs with $H \leq_R G$ and let $\Phi \in \hom_R(G,H)$. The \textit{stability relation} for $\Phi$, written $\sim_{\Phi}$, is the equivalence relation on $V(G)$ defined as follows: for $I \in V(H)$ and $I_1, I_2 \in (\partial \Phi)^{-1}(I)$, $I_1 \sim_{\Phi} I_2$ if and only if for all $u \in L_{I}(H)$, there exists $v \in L_{t(u)}(H)$ such that $I_1 \cdot uv = I_2 \cdot uv$.
\end{defn}

\begin{lemma}
Let $G,H$ be graphs with $H \leq_R G$ and let $\Phi \in \hom_R(G,H)$. The stability relation $\sim_{\Phi}$ is a congruence with respect to $\Phi$.
\end{lemma}

\begin{proof}
Let $I \in V(H)$ and $I_1, I_2 \in (\partial \Phi)^{-1}(I)$ with $I_1 \sim_{\Phi} I_2$. Let $u \in L_I(H)$ and let $v \in L_{t(u)}(H)$. Since $I_1 \sim_{\Phi} I_2$, there exists $w \in L_{t(v)}(H)$ such that $I_1 \cdot uvw = I_2 \cdot uvw$. Therefore $I_1 \cdot u \sim_{\Phi} I_2 \cdot u$, so $\sim_{\Phi}$ is indeed a congruence.
\end{proof}

We now define \textit{synchronizing} right-resolvers in terms of the stability relation, then show in \cref{prop-stab-sync-equiv} that, at least in the strongly connected case, this definition is equivalent to a more obvious notion of synchronization for a right-resolver.

\begin{defn}[synchronizer]
Let $G,H$ be graphs with $H \leq_R G$ and let $\Phi \in \hom_R(G,H)$. We say that $\Phi$ is \textit{synchronizing}, or is a \textit{synchronizer}, if each fiber $(\partial \Phi)^{-1}(I)$, $I \in V(H)$, is a $\sim_{\Phi}$ class. We denote the set of synchronizers $G \to H$ by $\hom_S(G,H)$, and we write $H \leq_S G$ if $\hom_S(G,H) \neq \emptyset$.
\end{defn}

Note that $\sim_{\Phi}$ depends on $\Phi$ only through $S_{\Phi}$. That is, if $\Phi, \Phi'$ are such that $S_{\Phi} = S_{\Phi'}$, then $\sim_{\Phi} = \sim_{\Phi'}$.

\begin{prop}\label{prop-stab-sync-equiv}
Let $G,H$ be graphs with $H \leq_R G$ and let $\Phi \in \hom_R(G,H)$. Then $\Phi$ is synchronizing if and only if for every $I \in V(H)$, there exists $u \in L_I(H)$ with $|(\partial \Phi)^{-1}(I) \cdot u| = 1$.
\end{prop}

\begin{proof}
First, suppose that $\Phi$ is synchronizing and let $I \in V(H)$. If $| (\partial \Phi)^{-1}(I) | = 1$, then we are done. Otherwise, there exist at least two distinct states $I_1, I_2 \in (\partial \Phi)^{-1}(I)$, and $I_1 \sim_{\Phi} I_2$ since $(\partial \Phi)^{-1}(I)$ is a $\sim_{\Phi}$ class by assumption. Therefore, there exists $u_1 \in L_I(H)$ with $I_1 \cdot u_1 = I_2 \cdot u_1$. In particular, $|(\partial \Phi)^{-1}(I) \cdot u_1| < | (\partial \Phi)^{-1}(I)|$. Continuing inductively, we can produce a sequence of words $u_1, \dots, u_n$ such that $t(u_i) = s(u_{i+1})$ and $| (\partial \Phi)^{-1}(I) \cdot u_1 \cdots u_n | = 1$. This proves the first claim.

For the converse, let $I \in V(H)$ and $I_1, I_2 \in (\partial \Phi)^{-1}(I)$. We will show that $I_1 \sim_{\Phi} I_2$. Let $v \in L_I(H)$ be arbitrary. By the assumption that each fiber can be collapsed to a single state, let $w \in L_{t(v)}(H)$ be such that $| (\partial \Phi)^{-1}(t(v)) \cdot w  | = 1$. Then $| (\partial \Phi)^{-1}(I) \cdot vw| = 1$; in particular, $I_1 \cdot vw \sim_{\Phi} I_2 \cdot vw$. Therefore $(\partial \Phi)^{-1}(I)$ is a $\sim_{\Phi}$ class, and $\Phi$ is synchronizing.
\end{proof}

We now summarize the structure of stability, in the sense of its behaviour with respect to composition of right-resolvers. For the proof of \cref{thm-struct-sync-comp}, see \S \ref{subsec-stab-proofs}.

\begin{thm}\label{thm-struct-sync-comp}
Let $G,K,H$ be graphs with $H \leq_R K \leq_R G$. Let $\Psi \in \hom_R(G,K)$, $\Delta \in \hom_R(K,H)$, and let $\Phi = \Delta \circ \Psi$.

\begin{enumerate}
    \item The $\sim_{\Psi}$ classes in $V(G)$ are the intersections of $\sim_{\Phi}$ classes with $\partial \Psi$ fibers. In particular, $\Psi$ is synchronizing if and only if every $\partial \Psi$ fiber is contained in a $\sim_{\Phi}$ class.
    
    \item If $K = G/\sim_{\Phi}$ and $\Psi$ is the quotient map for $\sim_{\Phi}$, then $\Psi$ is synchronizing and $\sim_{\Delta}$ is trivial.
    
    \item If $\sim_{\Delta}$ is trivial, then $\sim_{\Phi} = \sim_{\Psi}$.
    
    \item $\Phi$ is synchronizing if and only if both $\Psi$ and $\Delta$ are synchronizing.
\end{enumerate}
\end{thm}

The following observation follows immediately from \cref{thm-struct-sync-comp}(4).

\begin{cor}
The relation $\leq_S$ is transitive, and is thus a partial order on the class of graphs (again, really isomorphism classes of graphs), refining the partial order $\leq_R$.
\end{cor}

\begin{conj}[$O(G)$ conjecture, Question 4.6 in \cite{amt-95-etds}]\label{main-conj}
Let $G$ be a strongly connected graph. Then the set of graphs $H$ with $H \leq_S G$ has a unique $\leq_S$-minimal element $O(G)$. 
\end{conj}

\begin{rmk}
This remark is intended for readers interested in algebraic or categorical perspectives on automata theory. Recall that for a complete DFA, or road colouring $\Phi \in \hom_R(G,M_D)$, where $G$ is a graph of constant out-degree $D$, the set $S_{\Phi}$ of transition maps forms a transformation semigroup under composition. Indeed, a complete DFA is essentially a finite transformation semigroup together with a choice of generators; this perspective is taken explicitly in \cite{acs-17-prim,bf-11-aaecc} and mentioned in \cite{go-81-ijm}, the first paper on the road problem after \cite{agw-77-tms}. 

For a general right-resolver $\Phi \in \hom_R(G,H)$, one could see $S_{\Phi}$ as the semigroup of transitions of a partial finite automaton (PFA), where a given transition is defined only on a single fiber. However, it is more helpful to see $S_{\Phi}$ as a \textit{semigroupoid} (equivalently, if the empty word is included, a small category). One reason is that, as we show in \S \ref{sec-comput}, it can be decided in polynomial time whether $\Phi$ is synchronizing, and the length of a word synchronizing a given fiber is bounded by a polynomial in $|V(G)|$. This is in contrast to the high level of complexity typical of related problems in subset synchronization and synchronization of PFAs \cite{mb-14-dlt,vorel-16-ijfcs}.

The reader may verify as an exercise, generalizing Cayley's theorem or specializing the Yoneda lemma, that every finite semigroupoid is isomorphic to $S_{\Phi}$ for some graphs $G,H$ (although possibly with sinks) and some $\Phi \in \hom_R(G,H)$, with appropriate generalizations to the infinite case. Moreover, just as every group is a quotient of a free group, $S_{\Phi}$ is a quotient of the free semigroupoid $L(H)$.
\end{rmk}


\subsection{Sufficient conditions for stability}

We now give a pair of sufficient conditions for nontrivial stability, both of which are used in the proof of the road colouring theorem and one of which is also used in \S \ref{sec-bunchy} to obtain a right-resolver with nontrivial stability on an almost bunchy graph.

The first condition involves a special case of the operation known as in-amalgamation (\cite{lm-95-intro}, \S 2.4):

\begin{lemma}\label{lemma-amalg-stab}
Let $G,H$ be graphs with $H \leq_R G$ and let $\Phi \in \hom_R(G,H)$ be a right-resolver. Let $I \in V(H)$. Suppose that there exist $I_1, I_2 \in (\partial \Phi)^{-1}(I)$ such that $|E_{I_1 J}(G)| = |E_{I_2 J}(G)|$ for all $J \in V(G)$. Then there exists $\Phi' \in \hom_R(G,H)$ such that $I_1 \sim_{\Phi'} I_2$.
\end{lemma}

\begin{proof}
We first claim that $F(I_1) = F(I_2)$, where we recall from \S \ref{sec-graph-homo} the notation $F(\cdot)$ for the set of follower states of a given state. Indeed, the assumption that $|E_{I_1 J}(G)| = |E_{I_2 J}(G)|$ for all $J \in V(G)$ implies in particular that, for any $J \in V(G)$, we have $|E_{I_1 J}(G)| > 0$, equivalently $J \in F(I_1)$, if and only $|E_{I_2 J}(G)| > 0$, equivalently $J \in F(I_2)$.

Let $F = F(I_1) = F(I_2)$. For each $J \in F$, choose a bijection $\Theta_J: E_{I_2 J}(G) \to E_{I_1 J}(G)$. Define $\Phi'$ as follows: $\partial \Phi' = \partial \Phi$, $\Phi'|_{E(G) \sm E_{I_2}(G)} = \Phi|_{E(G) \sm E_{I_2}(G)}$, and, for each $J \in F$, $\Phi'|_{E_{I_2 J}(G)} = \Phi|_{E_{I_1 J}(G)} \circ \Theta_J$. That is, $\Phi'$ agrees with $\Phi$ on states and on all edges with initial state other than $I_2$, but may disagree with $\Phi$ on the outgoing edges from $I_2$, specifically by permutations of parallel edges. Since $\Phi'|_{E_{I_2}(G)}: E_{I_2}(G) \to E_I(G)$ is a bijection (being a composition of bijections), $\Phi'$ is indeed right-resolving.

By the construction of $\Phi'$, for any $a \in E_I(G)$, we have $I_1 \cdot a = I_2 \cdot a$. Any $w \in L_I(H)$ is of the form $w = au$ with $a \in E_I(G)$, so $I_1 \cdot w = I_2 \cdot w$. Therefore $I_1 \sim_{\Phi'} I_2$.
\end{proof}

In \cref{lemma-amalg-stab}, the states $I_1, I_2$ are said to be \textit{in-amalgamated} by the operation $G \to G/\sim_{\Phi'}$; the inverse operation is known as \textit{in-splitting}. The lemma shows in particular that no fiber of a $\leq_S$-minimal graph $G$ over $M(G)$ has two states that can be in-amalgamated. Trahtman applies a special case of this fact to graphs of constant out-degree, and we follow his line of application; see the first paragraph of the proof of \cref{thm-bfc-cyc-bunch}, found in \S \ref{subsec-traht-app}.

The second sufficient condition, given in \cref{lemma-min-diff-stab}, concerns \textit{minimal images}:

\begin{defn}[minimal image]
Let $G,H$ be graphs with $H \leq_R G$, and let $\Phi \in \hom_R(G,H)$. A \textit{minimal image} is a set of the form $U = (\partial \Phi)^{-1}(I) \cdot u$ for some $I \in V(H)$ and $u \in L_I(H)$, such that $|U \cdot v| = |U|$ for any $v \in L_{t(u)}(H)$. 
\end{defn}

\begin{rmk}
Let $G,H$ be graphs with $H \leq_R G$, and let $\Phi \in \hom_R(G,H)$. For any $I \in V(H)$, any $u \in L_{I}(H)$, and any $v \in L_{t(u)}(H)$, if $U = (\partial \Phi)^{-1}(I) \cdot u$, we clearly have $| U \cdot v | \leq | U |$. If there exists $v \in L_{t(u)}(H)$ such that this inequality is strict, then $|U|$ is not minimal, i.e. $U$ is not a minimal image. This is the reason for the term.

For a right-resolver $\Phi$ on a strongly connected graph, all minimal images have the same size, which is called the \textit{degree} of $\Phi$, and a word with minimal image is called a \textit{magic word}. See \cite{amt-95-etds}, \S 9.1 for a treatment of degrees, using symbolic dynamics. In this paper, we only need a small fragment of the theory of degree, which we establish in a self-contained way with no connectedness assumptions, using the properties of stability.
\end{rmk}

We use minimal images to give a criterion for stability that can be seen as a pairwise version of \cref{prop-stab-sync-equiv}.

\begin{lemma}\label{lemma-stab-sync-magic}
Let $G,H$ be graphs with $H \leq_R G$, and let $\Phi \in \hom_R(G,H)$. For $I \in V(H)$ and $I_1, I_2 \in (\partial \Phi)^{-1}(I)$, we have $I_1 \sim_{\Phi} I_2$ if and only if $I_1 \cdot u = I_2 \cdot u$ for every word $u \in L_I(H)$ such that $(\partial \Phi)^{-1}(I) \cdot u$ is a minimal image.
\end{lemma}

\begin{proof}
Let $I \in V(H)$ and let $I_1, I_2 \in (\partial \Phi)^{-1}(I)$. First suppose that $I_1 \sim_{\Phi} I_2$, and let $u \in L_I(H)$. If $I_1 \cdot u \neq I_2 \cdot u$, then let $v \in L_{t(u)}(H)$ be such that $I_1 \cdot uv = I_2 \cdot uv$. Then $| (\partial \Phi)^{-1}(I) \cdot uv | < | (\partial \Phi)^{-1}(I) \cdot u|$, so $(\partial \Phi)^{-1}(I) \cdot u$ is not a minimal image. Contrapositively, if $(\partial \Phi)^{-1}(I) \cdot u$ is a minimal image, then $I_1 \cdot u = I_2 \cdot u$.

Conversely, suppose that $I_1 \cdot u = I_2 \cdot u$ for every word $u \in L_I(H)$ such that $(\partial \Phi)^{-1}(I) \cdot u$ is a minimal image. Let $r = \min_{u \in L_I(H)} |(\partial \Phi)^{-1}(I) \cdot u|$ and let $u \in L_I(H)$. We claim that there exists $w \in L_{t(u)}(H)$ with $I_1 \cdot uw = I_2 \cdot uw$. Indeed, if $I_1 \cdot u \neq I_2 \cdot u$, then $(\partial \Phi)^{-1}(I) \cdot u$ is not a minimal image, so there exists $v \in L_{t(u)}(H)$ such that $|(\partial \Phi)^{-1}(I) \cdot uv | < |(\partial \Phi)^{-1}(I) \cdot u |$. We can thus inductively extend $v$ to obtain the desired $w$, so indeed $I_1 \sim_{\Phi} I_2$.
\end{proof}

The following easy observation about minimal images is the main reason that our results toward the bunchy factor conjecture (the generalized road colouring theorem and the related result for almost bunchy graphs) require strong connectedness. 

\begin{lemma}\label{lemma-sc-min-img}
Let $G,H$ be strongly connected graphs with $H \leq_R G$, and let $\Phi \in \hom_R(G,H)$. Every minimal image for $\Phi$ has the same cardinality, and for every $I' \in V(G)$, there exists a minimal image $U$ with $I' \in U$. 
\end{lemma}

In the proof of the generalized road colouring theorem, we apply \cref{lemma-sc-min-img} both directly and via \cref{lemma-min-diff-stab}. The proof of \cref{lemma-min-diff-stab} is adapted from the proof of Lemma 10.4.4 in \cite{bpr-09-ca}.

\begin{prop}\label{lemma-min-diff-stab}
Let $G, H$ be strongly connected graphs with $H \leq_R G$ and let $\Phi \in \hom_R(G,H)$. Let $I, J \in V(H)$ and let $u_1, u_2 \in L_{IJ}(H)$ be such that $U_i = (\partial \Phi)^{-1}(I) \cdot u_i$ are minimal images. If $|U_1 \Delta U_2| = 2$, say $U_1 \Delta U_2 = \{ J_1, J_2  \}$ (where $\Delta$ denotes the symmetric difference), then $J_1 \sim_{\Phi} J_2$.
\end{prop}

\begin{proof}
Let $r = |U_1| = |U_2|$. Suppose without loss of generality that $J_i \in U_i$, and let $U_0 = U_1 \cap U_2$, so that $U_i = U_0 \cup \{ J_i \}$. For any $v \in L_J(H)$, we have $(U_1 \cup U_2) \cdot v = (U_0 \cdot v) \cup (\{ J_1, J_2 \} \cdot v)$. We must have $|U_0 \cdot v| = |U_0| = r-1$ and $J_i \cdot v \notin U_0 \cdot v$, since otherwise we would have $|(\partial \Phi)^{-1}(I) \cdot u_i v| = |U_i \cdot v| < r$, contradicting the minimality assumption. Therefore $r-1 + | \{ J_1, J_2 \} \cdot v| = |(U_1 \cup U_2) \cdot v|$. Note that $|(U_1 \cup U_2) \cdot v| \in \{ r,r+1 \}$. 

Let $v \in L_J(H)$ be such that $(\partial \Phi)^{-1}(J) \cdot v$ is a minimal image. By \cref{lemma-stab-sync-magic}, to show that $J_1 \sim_{\Phi} J_2$, we need to show that $| \{ J_1, J_2 \} \cdot v| = 1$, or equivalently $|(U_1 \cup U_2) \cdot v| \leq r$. By strong connectedness and \cref{lemma-sc-min-img}, we have $|(\partial \Phi)^{-1}(J) \cdot v| = r$. But $(U_1 \cup U_2) \cdot v \se (\partial \Phi)^{-1}(J) \cdot v$, so indeed $|(U_1 \cup U_2) \cdot v| \leq r$.
\end{proof}


\section{The $O(G)$ conjecture and the road problem}\label{sec-og-road}

\subsection{Generalization of the road colouring theorem}

We first introduce the class of graphs involved in the theorem. A \textit{bunch} in a graph $G$ is a state $I \in V(G)$ with $|F(I)| = 1$. (This terminology, introduced in \cite{ant-09-ijm} and used also in \cite{bpr-09-ca}, is the origin of our term \textit{bunchy}, introduced in \S \ref{sec-bunchy}.) A strongly connected graph in which every state is a bunch is a \textit{cycle of bunches}. Let $M$ be a cycle of bunches with $V(M) = \{ I_0, \dots, I_{p-1} \}$, where $F(I_i) = \{ I_{i+1} \}$ (subscripts indexing states in a cycle of length $p$ should be read modulo $p$ throughout), and let $D_i = |E_{I_i}(M)|$. Note that $M$ is $\leq_R$-minimal if and only if the sequence of out-degrees $D_0, \dots, D_{p-1}$ is not a cyclic shift of a sequence obtained by concatenating a strictly shorter sequence with itself more than once. 

Let $M$ be a $\leq_R$-minimal cycle of bunches. Let $O_{M,q}$ be the cycle of bunches in which the sequence of out-degrees consists of $q$ cyclic repetitions of $D_0, \dots, D_{p-1}$. Note that $O_{M,1} = M$. Observe that, for a strongly connected graph $G$ with $M=M(G)$ a cycle of bunches, if $q = \per(G)/\per(M)$ and $H$ is a cycle of bunches with $H \leq_S G$, then $H = O_{M,q}$. Let $O_{D,p} = O_{M_D,p}$ be the cycle of bunches of period $p$ and constant out-degree $D$. Note that $O_{D,1} = M_D$. For a strongly connected, aperiodic graph $G$ of constant out-degree $D$, a synchronizer $G \to M_D$ is precisely a synchronizing road colouring of $G$. 

\begin{thm}[Trahtman, \cite{ant-09-ijm}]\label{thm-rct-aper}
Let $G$ be a strongly connected, aperiodic graph of constant out-degree $D$. Then $M_D \leq_S G$.
\end{thm}

\begin{thm}[Béal-Perrin \cite{bp-14-dam}, Budzban-Feinsilver \cite{bf-11-aaecc}]\label{thm-prct}
Let $G$ be a strongly connected graph of constant out-degree $D$ and period $p$. Then $O_{D,p} \leq_S G$.
\end{thm}

We prove the following generalization:

\begin{thm}\label{thm-bfc-cyc-bunch}
Let $G$ be a strongly connected graph such that $M(G)$ is a cycle of bunches. Let $q = \per(G)/\per(M(G))$. Then $O_{M(G),q} \leq_S G$.
\end{thm}

The proof (see \S \ref{sec-traht}) follows that of Theorems \ref{thm-rct-aper} and \ref{thm-prct}. The strategy is to show that if $G$ is not itself a cycle of bunches, then there exists $\Phi \in \hom_R(G,M(G))$ with $\sim_{\Phi}$ nontrivial, by constructing two minimal images that differ by a pair and applying \cref{lemma-min-diff-stab}. A very similar strategy is used to prove the bunchy factor conjecture for almost bunchy graphs (\cref{cor-alm-og-bunchy}), the (substantial) difference being the different techniques used to obtain the requisite pair of minimal images.

\subsection{The $O(G)$ conjecture implies the road colouring theorem}

We now recall the sense in which the $O(G)$ conjecture was first understood to relate to the road problem. Although the road colouring theorem clearly implies the $O(G)$ conjecture for strongly connected, aperiodic graphs of constant out-degree, the converse implication may not be apparent. Indeed, the $O(G)$ conjecture asserts that $O(G)$ is well-defined for every strongly connected graph, but does not immediately say how to compute $O(G)$, whereas the road colouring theorem explicitly specifies the form of $O(G)$ for the graphs $G$ to which it applies. However, the $O(G)$ conjecture does imply the road colouring theorem, via a key result from \cite{agw-77-tms}, for which we require a definition.

\begin{defn}[higher edge graph]\label{defn-higher-edge}
Let $G$ be a graph. For $k \geq 2$, the $k$-th \textit{higher edge graph} of $G$ is the graph $G^{[k]}$ with edge set consisting of edge paths $e_1 e_2 \cdots e_{k-1} e_k$ of length $k$ in $G$, and states given by $s(e_1 e_2 \cdots e_{k-1} e_k) = e_1 e_2 \cdots e_{k-1}$, $t(e_1 e_2 \cdots e_{k-1} e_k) = e_2 \cdots e_{k-1} e_k$. We define $G^{[1]} = G$.
\end{defn}

It is a standard result (\cite{lm-95-intro}, Chapter 2) that $G \leq_S G^{[k]}$ for any strongly connected graph $G$ and any $k \geq 1$. In our terminology, Adler-Goodwyn-Weiss showed the following:

\begin{lemma}[\cite{agw-77-tms}, Lemma 4]\label{lemma-agw-higher-edge}
Let $G$ be a strongly connected, aperiodic graph of constant out-degree $D$. Then for all sufficiently large $k$, we have $M_D \leq_S G^{[k]}$.
\end{lemma}

Together with an easy observation about partially ordered sets, the Adler-Goodwyn-Weiss result shows that the $O(G)$ conjecture implies the road colouring theorem.

\begin{lemma}\label{lemma-poset}
Let $(\cP, \preceq)$ be a partially ordered set such that, for any $y \in \cP$, there exists a unique $\preceq$-minimal element $O(y) \preceq y$. If $x \preceq y$, then $O(x) = O(y)$.
\end{lemma}

\begin{prop}[\cite{amt-95-etds}]\label{prop-amt-og-implies-road}
Suppose that the $O(G)$ conjecture is true. Let $G$ be a strongly connected, aperiodic graph of constant out-degree $D$. Then $M_D \leq_S G$.
\end{prop}

\begin{proof}
Let $k$ be large enough that $M_D \leq_S G^{[k]}$, by \cref{lemma-agw-higher-edge}. Then in fact $M_D = O(G^{[k]})$. Since $G \leq_S G^{[k]}$ as well, the result follows by \cref{lemma-poset}.
\end{proof}

\section{Bunchiness}\label{sec-bunchy}

In this section, we define and characterize the classes of bunchy and almost bunchy graphs, and demonstrate the importance of bunchy graphs to the structural properties of right-resolvers. 

\subsection{Bunchy and almost bunchy graphs}

We recall from Theorem \ref{thm-mg-sigma}, for a graph $G$, the notation $\Sigma_G: V(G) \to V(M(G))$ for the unique state map among right-resolvers $G \to M(G)$. 

\begin{defn}[bunchy states and graphs]\label{defn-bunchy}
Let $G$ be a graph. We say that a state $I \in V(G)$ is \textit{bunchy} if $\Sigma_G|_{F(I)}: F(I) \to F(\Sigma_G(I)) \se V(M(G))$ is a bijection. We say that $G$ is bunchy if every $I \in V(G)$ is bunchy. We say that $G$ is \textit{almost bunchy} if for each $I,J \in V(M)$, there exists at most one $I' \in \Sigma_G^{-1}(I)$ such that $|F(I') \cap \Sigma_G^{-1}(J)| \geq 2$.
\end{defn}

\begin{rmk}
The definition of almost bunchiness means that for every ordered pair of $\Sigma_G$ fibers in $G$, say the fibers of states $I,J \in V(M)$, there is at most one state in the fiber of $I$ that does not ``look bunchy'', in the sense that it has non-parallel outgoing edges into the fiber of $J$. In other words, an almost bunchy graph almost satisfies the conditions for bunchiness, but an exception is allowed for each ordered pair of fibers.
\end{rmk}

The following is evident but we state it explicitly for reference:

\begin{lemma}\label{lemma-bunchy-factor-princ}
The classes of bunchy and almost bunchy graphs are closed under right-resolvers. Moreover, if $G$ is a bunchy graph and $C$ is a principal subgraph of $G$ with $M(C) = M(G)$, then $C$ is also bunchy.
\end{lemma}

\begin{rmk}
We briefly discuss examples of bunchy and almost bunchy graphs. The only strongly connected bunchy graphs of constant out-degree are the cycles of bunches. For a given $\leq_R$-minimal cycle of bunches $M$ with sequence of out-degrees $D_0, \dots, D_{p-1}$, the only strongly connected bunchy graphs $G$ with $M(G) = M$ are the graphs $O_{M,q}$ introduced in the previous section. 

A strongly connected almost bunchy (but not bunchy) graph of constant out-degree is a graph with a unique non-bunchy state, i.e. a state $I$ with $|F(I)| \geq 2$, together with a path from each element of $F(I)$ back to $I$. One example that has been considered in the literature is the graph $W_n$ studied in \cite{agv-10-lncs}, first discussed in \cite{wiel-50-mz}, and of interest due to its slow synchronization.

An almost bunchy graph $G$ can have at most $|V(M(G))|^2$ non-bunchy states, one for each ordered pair of $\Sigma_G$ fibers. One way to obtain an almost bunchy graph is to start with a bunchy graph and perform a sequence of in-splittings (recall \cref{lemma-amalg-stab}, and see also \cite{lm-95-intro}, \S 2.4), but not all in-splittings will preserve almost bunchiness, and not all almost bunchy graphs arise this way.
\end{rmk}

Bunchy and almost bunchy graphs are characterized in terms of automorphisms, with an important uniqueness consequence for the sets of transition maps induced on them by right-resolvers:

\begin{prop}\label{prop-alm-bunchy-equiv}
A graph $G$ is almost bunchy if and only if there is a unique right-resolver $G \to M(G)$ up to permutations of parallel edges: that is, if and only if, for any $\Phi_1, \Phi_2 \in \hom_R(G,M(G))$, there exist $\sigma \in P(G), \tau \in P(M(G))$ such that $\Phi_1 = \tau \circ \Phi_2 \circ \sigma$. Moreover, $G$ is bunchy if and only if we can take $\tau = \id$ regardless of $\Phi_1, \Phi_2$.
\end{prop}

\begin{prop}\label{cor-alm-stab-unique}
Let $G$ be an almost bunchy graph. Let $\Phi_1, \Phi_2 \in \hom_R(G,M(G))$. Then $S_{\Phi_1} = S_{\Phi_2}$. In particular, $\sim_{\Phi_1} \, = \, \sim_{\Phi_2}$.
\end{prop}

For the proofs of Propositions \ref{prop-alm-bunchy-equiv} and \ref{cor-alm-stab-unique}, see \S \ref{sec-bunchy-details}. The following definition is now justified.

\begin{defn}
Let $G$ be an almost bunchy graph. We denote by $\sim_G$ the unique relation on $V(G)$ with $\sim_G = \sim_{\Phi}$ for any $\Phi \in \hom_R(G,M(G))$.
\end{defn}


\subsection{Proof of the $O(G)$ conjecture in the bunchy case}

We now resolve the $O(G)$ conjecture in the almost bunchy case (which includes the bunchy case). This extends Corollary 4.3 in \cite{amt-95-etds}, which resolves the conjecture for graphs $G$ such that $M(G)$ has no parallel edges (so $G$ is trivially bunchy). The proof here is quite different from the proof in the no-parallel-edges case, and yields a polynomial-time algorithm (Algorithm \ref{alg-construct-stab}) for constructing $O(G)$.

\begin{thm}\label{thm-alm-og-exists}
Let $G$ be an almost bunchy graph and let $H \leq_S G$. If $H$ is $\leq_S$-minimal, then $H = G/\sim_G$. In particular, the set $\{ K \, | \, K \leq_S G  \}$ has a unique $\leq_S$-minimal element $O = G/\sim_G$.
\end{thm}

\begin{proof}
Let $\Psi \in \hom_S(G,H)$ and $\Delta \in  \hom_R(H,M)$. We can naturally identify $V(H) = V(G) / \sim_{\Psi}$ by \cref{lemma-fiber-determ-img}. The hypothesis that $H$ is $\leq_S$-minimal implies that $\sim_{\Delta}$ is trivial, so $\sim_{\Psi} \, = \, \sim_{\Delta \circ \Psi}$ by \cref{thm-struct-sync-comp}(3). By \cref{cor-alm-stab-unique}, we have $\sim_{\Delta \circ \Psi} \, = \, \sim_G$. We can thus identify $V(H)$ with $V(G) / \sim_G \, = V(G/ \sim_G)$. Thus $H = G/ \sim_G$ by a second application of \cref{lemma-fiber-determ-img}.
\end{proof}

In the strongly connected case, we can apply \cref{lemma-min-diff-stab}, which is also used in the proof of the road colouring theorem, to say more.

\begin{prop}\label{prop-bfc-alm}
Let $G$ be a strongly connected almost bunchy graph. If $G$ is not bunchy, then $\sim_G$ is nontrivial.
\end{prop}

\begin{proof}
Let $I,J \in V(M(G))$ and $I' \in \Sigma_G^{-1}(I)$ with $|F(I') \cap \Sigma_G^{-1}(J)| \geq 2$. Let $J_1, J_2 \in F(I') \cap \Sigma_G^{-1}(J)$, $J_1 \neq J_2$, and let $e_i \in E_{I' J_i}(G)$. Let $\Phi \in \hom_R(G,M(G))$ and let $a_i = \Phi(e_i)$. By strong connectedness and \cref{lemma-sc-min-img}, there exists a minimal image $U \se \Sigma_G^{-1}(I)$ with $I' \in U$. Let $U_0 = U \sm \{ I' \}$. Then $U_0 \cdot a_1 = U_0 \cdot a_2$ since $G$ is almost bunchy. Moreover, $J_i = I' \cdot a_i \notin U_0 \cdot a_i$ (otherwise, minimality would be contradicted). Thus $(U \cdot a_1) \Delta (U \cdot a_2) = \{ J_1, J_2 \}$. By \cref{lemma-min-diff-stab}, $J_1 \sim_{\Phi} J_2$. 
\end{proof}

\begin{cor}\label{cor-alm-og-bunchy}
Let $G$ be a strongly connected almost bunchy graph. Then $O(G)$ is bunchy.
\end{cor}

\begin{proof}
If $|V(G)| = 1$ then the claim is clearly true. Suppose that the conclusion is true for all almost bunchy $G$ with $|V(G)| \leq_R N$, and let $G$ be almost bunchy with $|V(G)| = N+1$. If $G$ is bunchy, then $O(G)$ is clearly bunchy by \cref{lemma-bunchy-factor-princ}. If $G$ is not bunchy, then $\sim_{G}$ is nontrivial by \cref{prop-bfc-alm}, so $|V(G/\sim_{G})| \leq_R N$. Moreover, since $G/\sim_G \, \leq_S G$, it follows that $O(G) = O(G/\sim_G)$ is bunchy by the inductive hypothesis and \cref{lemma-bunchy-factor-princ}.
\end{proof}


\subsection{Universal property of the fiber product}\label{subsec-fiber-univ-prop}

We recall a standard construction known as the fiber product, and derive several new properties. Chief among these is the one exhibited in \cref{thm-univ-prop}, which is an analogue of the universal property often enjoyed by the fiber product, or pullback, in other categories (see e.g. \cite{er-17-dover}, Definition 3.1.15 and subsequent discussion).

\begin{defn}[fiber product]\label{defn-fiber}
Let $H_1, H_2,K$ be graphs and let $\Psi_i: H_i \to K$ be graph homomorphisms. The \textit{fiber product} of $\Psi_1, \Psi_2$ is the graph $P = H_1 \times_{\Psi_1, \Psi_2} H_2$ where
\begin{align*}
    V(P) &= \bigsqcup_{I \in V(K)} (\partial \Psi_1)^{-1}(I) \times (\partial \Psi_2)^{-1}(I) \\
    E(P) &= \{ (e_1, e_2) \, | \, e_i \in E(H_i), \, \Psi_1(e_1) = \Psi_2(e_2)   \}
\end{align*}
together with the coordinate projections $\hat{\Psi}_i: P \to H_i$. We write $\Psi_P = \Psi_i \circ \Hat{\Psi}_i : P \to K$.
\end{defn}

\begin{rmk}
To see that $\Psi_1 \circ \Hat{\Psi}_1 = \Psi_2 \circ \Hat{\Psi}_2$ and thus that $\Phi_P$ is well-defined, note that for every $(I_1, I_2) \in V(P)$ and every $(e_1, e_2) \in E_{(I_1, I_2)}(P)$, we have, by the definition of $P$,
\[
\Psi_1 \circ \Hat{\Psi}_1 (e_1, e_2) = \Psi_1(e_1) = \Psi_2(e_2) = \Psi_2 \circ \Hat{\Psi}_2 (e_1, e_2)
\]
\end{rmk}

\begin{rmk}
Observe that the $\Hat{\Psi}_i$ are surjective (respectively, right-resolving) when the $\Psi_i$ are surjective (respectively, right-resolving). Moreover, if $C$ is a principal subgraph of $P$ such that the restricted state maps $\partial \Hat{\Psi}_i|_{V(C)} : V(C) \to V(H_i)$ are surjective, then $H_i \leq_R C$, indeed $\Hat{\Psi}_i|_C \in \hom_R(C,H_i)$. In particular, this condition is satisfied if the $H_i$ are strongly connected and $C$ is a principal component of $P$. 
\end{rmk}

\begin{rmk}
Often the convention is taken that $V(P) = V(H_1) \times V(H_2)$. However, all of the elements of the full Cartesian product that are not elements of $V(P)$, as we have defined it, would be isolated states, and in particular would be sinks. Our definition has the feature that the fiber product of two sink-free graphs (or rather, of two right-resolvers defined on such graphs) is also sink-free.
\end{rmk}

We now state the universal property of the fiber product. Compare with a similar diagram in \cite{amt-95-etds} (p. 289). See \S \ref{subsec-univ-prop-proof} for the proof.

\begin{thm}\label{thm-univ-prop}
Let $H_1, H_2$ be bunchy graphs with $M(H_1) = M(H_2) = M$. Let $\Psi_i \in \hom_R(H_i, M)$ be right-resolvers, and let $P = H_1 \times_{\Psi_1, \Psi_2} H_2$. Let $G$ be a common right-resolving extension of $H_1, H_2$ via $\Phi_i \in \hom_R(G, H_i)$. Then there exist a principal subgraph $C$ of $P$ and right-resolvers $\Delta_i \in \hom_R(G,C)$ such that $\Phi_i = \hat{\Psi}_i \circ \Delta_i$ and $\partial \Delta_1 = \partial \Delta_2$. In particular, $H_i \leq_R C \leq_R G$, with $\Hat{\Psi}_i|_C \in \hom_R(C,H_i)$.
\end{thm}

\begin{rmk}
The bunchiness hypothesis on the $H_i$ cannot be dropped, as the following construction illustrates. Let $G$ be a graph and let $\Phi_1, \Phi_2 \in \aut(G)$. In the notation of the the theorem, we will take $H_1 = H_2 = G$. (Recall that any automorphism is right-resolving.) Let $M = M(G)$. Let $\Psi_i \in \hom_R(G, M)$ and $P = G \times_{\Psi_1, \Psi_2} G$. Let $C$ be a principal subgraph of $P$ with $\Hat{\Psi}_i|_V(C)$ surjective, and let $\Delta_i \in \hom_R(G,C)$ with $\Phi_i = \Hat{\Psi}_i|_C \circ \Delta_i$. Then $\Hat{\Psi}_i|_C$ and $\Delta_i$ are isomorphisms, since they compose to an isomorphism. In particular, since $\Psi_1 \circ \Hat{\Psi}_1|_C = \Psi_2 \circ \Hat{\Psi}_2|_C$, we have $\Psi_1 = \Psi_2 \circ \tau$ where $\tau = \Hat{\Psi}_2|_C \, \circ \left( \Hat{\Psi}_1|_C  \right)^{-1}$ is an isomorphism. In other words, any two elements of $\hom_R(G, M)$ agree up to an automorphism of $G$. That condition always holds when $G$ is bunchy (see \cref{prop-alm-bunchy-equiv}), but fails in general.
\end{rmk}

We now give two applications of the universal property. The first, \cref{prop-bunchy-same-o}, is applied in \cref{prop-bunchy-implies-o}, which is the main motivation for the bunchy factor conjecture. See \S \ref{subsec-fiber-max-bunchy} for the proof of \cref{lemma-fiber-bunchy}.

\begin{lemma}\label{lemma-fiber-bunchy}
Let $H_1, H_2$ be bunchy graphs with $M(H_1) = M(H_2) = M$. Let $\Psi_i \in \hom_R(H_i, M)$, and let $P = H_1 \times_{\Psi_1, \Psi_2} H_2$. Then $P$ is bunchy. In particular, if $C$ is a principal subgraph of $P$ such that the restrictions $\partial \Hat{\Psi}_i|_{V(C)} : V(C) \to V(H_i)$ are surjective, then $C$ is bunchy.
\end{lemma}

\begin{prop}\label{prop-bunchy-same-o}
Let $G$ be a graph. Let $H_1, H_2 \leq_S G$ be bunchy. Then $O(H_1) = O(H_2)$, i.e. $G$ has at most one $\leq_S$-minimal bunchy synchronizing factor.
\end{prop}

\begin{proof}
Let $M = M(G)$, let $\Phi_i \in \hom_S(G,H_i)$, and let $\Psi_i \in \hom_R(H_i, M)$. Let $P = H_1 \times_{\Psi_1, \Psi_2} H_2$. By \cref{thm-univ-prop}, there is a principal subgraph $C$ of $P$ admitting $\Delta_i \in \hom_R(G,C)$ such that $\Phi_i = \hat{\Psi}_i|_C \circ \Delta_i$. Since each $\Phi_i$ is synchronizing, each restriction $\hat{\Psi}_i|_C$ is synchronizing as well, so $H_i \leq_S C$. Since $C$ is bunchy by \cref{lemma-fiber-bunchy}, we know that $O(C)$ is well-defined, and thus, by \cref{lemma-poset}, we have $O(H_1) = O(C) = O(H_2)$ as claimed.
\end{proof}

For the second application of the universal property, recall that the only strongly connected bunchy graphs of constant out-degree are the cycles of bunches. In particular, by the periodic road colouring theorem, for any strongly connected graph $G$ of constant out-degree $D$ and period $p$, the unique maximal bunchy right-resolving factor of $G$, namely $O_{D,p}$, is a synchronizing factor of $G$ (and is indeed equal to $O(G)$). We now show that every graph $G$ has a unique maximal bunchy right-resolving factor $B(G)$. The construction is similar to that of the auxiliary graph $\Tilde{G}$ in \cite{amt-95-etds}, \S 5. See \S \ref{subsec-fiber-max-bunchy} for the proof, as well as an explicit construction of $B(G)$ yielding a polynomial-time algorithm (Algorithm \ref{alg-construct-bg}).

\begin{prop}\label{prop-bg-exists}
Let $G$ be a graph. 

\begin{enumerate}[label=(\arabic*)]
    \item The set $\{ H \leq_R G \, | \, H \text{ is bunchy}  \}$ has a unique $\leq_R$-maximal element $B=B(G)$.
    
    \item Let $H \leq_R G$ be bunchy and $\Phi \in \hom_R(G,H)$. Then $\Phi$ factors through $B$, i.e. there exist $\Delta \in \hom_R(G,B)$, $\Theta \in \hom_R(B,H)$ such that $\Phi = \Theta \circ \Delta$. 
\end{enumerate}
\end{prop}


\section{The $O(G)$ conjecture and bunchy synchronizing factors}\label{sec-conjec}

As we have seen, the $O(G)$ conjecture holds for strongly connected graphs $G$ such that $M(G)$ is a cycle of bunches, and for almost bunchy graphs (including the bunchy graphs) by \cref{thm-alm-og-exists}. Moreover, for strongly connected almost bunchy graphs, and strongly connected graphs that factor onto cycles of bunches, we know that there is a bunchy synchronizing factor, which we show inductively by assuming non-bunchiness and obtaining a right-resolver with a nontrivial stability relation. It seems plausible that, if the $O(G)$ conjecture is true, then it can be proven by a similar approach: assume non-bunchiness, find a right-resolver with nontrivial stability relation, recursively find a bunchy synchronizing factor, and apply \cref{prop-bunchy-same-o}. The next proposition gives several equivalent formulations of the hypothesis that this approach can be made to work. See \S \ref{subsec-fiber-fact-ext} for the proof.

\begin{prop}\label{prop-bfc-equiv-new}
The following statements are equivalent.

\begin{enumerate}[label=(\arabic*)]
    \item Any strongly connected $\leq_S$-minimal graph is bunchy.

    \item For any strongly connected graph $G$, there exists some bunchy $H \leq_S G$.
    
    \item For any non-bunchy strongly connected graph $G$, there exists some $\Phi \in \hom_R(G,M(G))$ with $\sim_{\Phi}$ nontrivial.
    
    \item For any strongly connected graph $G$, $B(G) \leq_S G$.
\end{enumerate}
\end{prop}

\begin{conj}[bunchy factor conjecture]
The assertions in \cref{prop-bfc-equiv-new} are true.
\end{conj}

\begin{prop}\label{prop-bunchy-implies-o}
The bunchy factor conjecture implies the $O(G)$ conjecture.
\end{prop}

\begin{proof}
Let $G$ be a strongly connected graph and let $A = \{ H \, | \, H \leq_S G, \, H \text{ is $\leq_S$-minimal} \}$. Clearly $|A| \geq 1$. By hypothesis, every element of $A$ is bunchy. By \cref{prop-bunchy-same-o}, $|A| \leq 1$, so $A$ has a single element, namely $O(G)$.
\end{proof}

Observe that the bunchy factor conjecture is a straightforward generalization of the road problem. As discussed above, the $O(G)$ conjecture was already known to imply the road colouring theorem, via the higher-edge result from \cite{agw-77-tms}. By contrast, the bunchy factor conjecture implies the road colouring theorem more directly, without reference to \cite{agw-77-tms}.

\section{Computing with right-resolvers}\label{sec-comput}

We now discuss the computational problems of constructing $O(G_1), O(G_2)$, for input graphs $G_1, G_2$ such that the $O(G_i)$ are known to exist, and deciding whether the $O(G_i)$ are isomorphic. Although one could apply generic graph isomorphism algorithms, which are efficient in practice (Theorem 1 in \cite{miller-79-jcss} gives a polynomial-time reduction from directed to undirected graph isomorphism, and see \cite{babai-18-icm} for a survey of the state of the art), it is desirable to have a polynomial-time algorithm, in particular one that only uses constructions involved in the theory of right-resolvers and synchronization. We do not attempt detailed complexity analyses, noting only that all of the procedures we give can be easily seen to run in polynomial time.

\subsection{Basic routines}\label{subsec-comput-basic}

In \cite{amt-95-etds}, a polynomial-time algorithm is given for computing $M(G)$, along with $\Sigma_G$ and a total ordering of $V(M(G))$ such that any graph isomorphism $M(G) \to M(H)$ must preserve the order of the states. Deciding whether $M(G)$, $M(H)$ are isomorphic is therefore no harder than constructing them. Moreover, we can use $\Sigma_G$ to construct right-resolvers, as follows.

\begin{alg}\label{alg-construct-rr}
Construct a right-resolver from a graph to its minimal right-resolving factor.

\begin{enumerate}
    \item Input: a graph $G$.
    
    \item Construct $M(G)$ and $\Sigma_G$.
    
    \item For each $I,J \in V(M(G))$:
    
    \begin{enumerate}[label=\arabic*.]
        \item Choose a total ordering on $E_{IJ}(M)$.
        
        \item For each $I' \in \Sigma_G^{-1}(I)$:
        
        \begin{enumerate}[label=\arabic*.]
            \item Choose a total ordering on $\bigcup_{J' \in \Sigma_G^{-1}(J)} E_{I'J'}(G)$, i.e. the edges $e \in E_{I'}(G)$ with $\Sigma_G(t(e)) = J$.
            
            \item For each $e \in E_{I'}(G)$ with $\Sigma_G(t(e)) = J$, record as $\Phi(e)$ the edge in $E_{IJ}(M)$ with the same position in the total ordering of $E_{IJ}(M)$ as $e$ has in $\bigcup_{J' \in \Sigma_G^{-1}(J)} E_{I'J'}(G)$.
        \end{enumerate}
    \end{enumerate}
    
    \item Return: $\Phi$.
\end{enumerate}

\end{alg}

There are also obvious polynomial-time procedures for constructing the fiber product of two right-resolvers, and for determining whether there is a path from one given state to another (in a graph that may not be strongly connected). With these basic routines, we can construct the stability relation of a right-resolver in polynomial time, as follows.

\begin{alg}\label{alg-construct-stab}
Construct the stability relation of a right-resolver.

\begin{enumerate}
    \item Input: graphs $G, H$, and $\Phi \in \hom_R(G,H)$.
    
    \item Construct $P = G \times_{\Phi, \Phi} G$.
    
    \item Populate the set $U$ of states $(I_1, I_2) \in V(P)$ with no outgoing path to the diagonal in $V(P)$.
    
    \item Populate and output: the set $\sim_{\Phi}$ of states $(I_1, I_2) \in V(P)$ with no outgoing path to $U$.
\end{enumerate}
\end{alg}

Recall that, by definition, $\Phi \in \hom_S(G,H)$ if and only if $H = G/\sim_{\Phi}$, i.e. the $\partial \Phi$-fibers are precisely the $\sim_{\Phi}$ classes. Since this is easy to check, Algorithm \ref{alg-construct-stab} can be used to decide whether $\Phi$ is synchronizing.

A similar procedure can be used to construct the maximum bunchy factor $B(G)$, following the construction in \cref{prop-construct-bg}:

\begin{alg}\label{alg-construct-bg}
Construct the maximum bunchy right-resolving factor of a graph.

\begin{enumerate}
    \item Input: a graph $G$.
    
    \item Construct $M(G)$, along with the quotient map $\Sigma_G: V(G) \to V(M)$.
    
    \item Construct a graph $H$ with the following data:
    \begin{align*}
        V(H) &= V(G) \times V(G) \\
        |E_{(I_1, I_2)(J_1, J_2)}(H)| &=
        \begin{cases}
          1, & \Sigma_G(I_1) = \Sigma_G(I_2), \Sigma_G(J_1) = \Sigma_G(J_2), \text{ and } J_i \in F(I_i) \\
          0, & \text{otherwise}
        \end{cases}
    \end{align*}

    \item Populate the set $\approx_0$ of pairs $(I_1, I_2)$ with a path in $H$ from the diagonal to $(I_1, I_2)$.
    
    \item Construct and output: the transitive closure $\approx$ of the relation $\approx_0$.
\end{enumerate} 

\end{alg}

Regarding step 4: referring to the description of $\approx_0$ in \cref{prop-construct-bg}, a path in $H$ from the diagonal to $(I_1, I_2)$ corresponds to a pair of paths $\gamma, \delta \in L(G)$ witnessing $I_1 \approx_0 I_2$.


\subsection{Decision procedures for common synchronizing factors and extensions}\label{subsec-comput-decide}

See \S \ref{subsec-fiber-fact-ext} for the proof of the following Proposition, which collects several similar statements relating common synchronizing factors and common synchronizing extensions. These results allow us to use fiber products of two graphs to decide questions about common factors of the graphs.

\begin{prop}\label{cor-air-fact-ext-mg}
Let $G_1,G_2$ be strongly connected graphs with $M(G_1) = M(G_2) = M$.

\begin{enumerate}
    \item If $G_1, G_2$ have a common synchronizing factor, then there exist $\Phi_i \in \hom_R(G_i, M)$ and a principal component $C$ of $Q = G_1 \times_{\Phi_1, \Phi_2} G_2$ such that $\Hat{\Phi}_i|_C \in \hom_S(C,G_i)$. 
    
    \item Assume the $O(G)$ conjecture. Then the converse holds in (1). That is, suppose that there exist $\Phi_i \in \hom_R(G_i, M)$ and a principal component $C$ of $Q = G_1 \times_{\Phi_1, \Phi_2} G_2$, such that $\Hat{\Phi}_i|_C \in \hom_S(C,G_i)$. Then $G_1, G_2$ have a common synchronizing factor, specifically $O(G_1) = O(C) = O(G_2)$.
    
    \item If the $G_i$ are bunchy, then the equivalence described in (1) and (2) holds unconditionally (i.e. without assuming any unproven conjectures).
    
    \item Assume the bunchy factor conjecture. Then we have $O(G_1) = O(G_2)$ if and only if, for the essentially unique $\Phi_i \in \hom_R(B(G_i), M)$, there is a principal component $C$ of $P = B(G_1) \times_{\Phi_1, \Phi_2} B(G_2)$ such that $\Hat{\Phi}_i|_C \in \hom_S(C, B(G_i))$.
\end{enumerate}

\end{prop}

Ashley-Marcus-Tuncel give a polynomial-time algorithm for deciding whether two strongly connected graphs have a common strongly connected synchronizing extension. Their algorithm relies on the construction of a graph that they call $\Tilde{G}$, from an input graph $G$; see Theorem 5.2 and Remark 5.10 in \cite{amt-95-etds}. If the $O(G)$ conjecture is true, then by \cref{cor-air-fact-ext-mg}(2), the Ashley-Marcus-Tuncel algorithm also decides whether $O(G_1) = O(G_2)$ for $G_1, G_2$ strongly connected. Without assuming any unproven conjectures, a negative result from the algorithm shows, by \cref{cor-air-fact-ext-mg}(1), that $G_1, G_2$ have no common synchronizing factor, while a positive result is inconclusive. 

For bunchy $G_1, G_2$, however, \cref{thm-alm-og-exists} and \cref{cor-air-fact-ext-mg}(3) show that it can be decided in polynomial time whether $O(G_1), O(G_2)$ are isomorphic without the Ashley-Marcus-Tuncel algorithm:

\begin{alg}\label{alg-decide-bunchy-og}
Decide whether $O(G_1), O(G_2)$ are isomorphic for $G_1, G_2$ strongly connected and bunchy.

\begin{enumerate}
    \item Input: strongly connected bunchy graphs $G_1, G_2$ such that $M(G_1) = M(G_2) = M$.
    
    \item Choose $\Phi_i \in \hom_R(G_i,M)$, using Algorithm \ref{alg-construct-rr}.
    
    \item Construct $Q = G_1 \times_{\Phi_1, \Phi_2} G_2$.
    
    \item For each principal component $C$, decide whether $\Hat{\Phi}_i|_C \in \hom_S(C, G_i)$, using Algorithm \ref{alg-construct-stab}. If so, halt, and return the result that $O(G_1) = O(G_2)$. 
    
    \item If no affirmative result is returned in step 4, then halt and return the result that $O(G_1) \neq O(G_2)$.
\end{enumerate}
\end{alg}

Regarding steps 4--5: \cref{cor-air-fact-ext-mg}(3) shows that we have $O(G_1) = O(G_2)$ if and only if $\Hat{\Phi}_i|_C \in \hom_S(C,G_i)$ for some principal component $C$ of $Q$. The same algorithm would work even if $G_1, G_2$ are not strongly connected but only bunchy, with ``principal component $C$'' replaced by ``principal subgraph $C$ such that each $\partial \Hat{\Phi}_i|_{V(C)}$ is surjective'', but it is not clear that the number of such subgraphs is bounded by a polynomial in $|V(G)|$.

Furthermore, if the bunchy factor conjecture is true, then isomorphism of $O(G_1), O(G_2)$ is equivalent, by \cref{prop-bfc-equiv-new}, to isomorphism of $O(B(G_1)), O(B(G_2))$, which can be decided without the use of the Ashley-Marcus-Tuncel $\Tilde{G}$ algorithm, using the following procedure. Note that this works even if we do not have an efficient way to find an element of $\hom_S(G_i, B(G_i))$. 

\begin{alg}\label{alg-bfc-iso-og}
Decide isomorphism of $O(G_1), O(G_2)$ (assuming the bunchy factor conjecture).

\begin{enumerate}
    \item Input: strongly connected graphs $G_1, G_2$.
    
    \item Construct $B(G_1)$ and $B(G_2)$, using Algorithm \ref{alg-construct-bg}.
    
    \item Decide whether $O(B(G_1)) = O(B(G_2))$, using Algorithm \ref{alg-decide-bunchy-og}.
    
    \item Output: the Boolean value of ``$O(B(G_1)) = O(B(G_2))$''.
\end{enumerate}

\end{alg}

Observe that the procedure of Algorithm \ref{alg-bfc-iso-og} is essentially what is described in \cref{cor-air-fact-ext-mg}(4).


\section{Proofs of structural results and additional details}\label{sec-proofs}

\subsection{Remarks on the proof of \cref{thm-mg-sigma}}\label{subsec-right-air-mg}

In this subsection, we revisit the first proof of \cref{thm-mg-sigma} given in \cite{amt-95-etds}. The proof of the uniqueness of $M(G)$, for a given graph $G$, is without issue, but the proof of the uniqueness of $\Sigma_G$ is not quite complete. That proof seems to assume that, for graphs $G, H$ with $H \leq_R G$, two right-resolvers $G \to H$ with distinct state maps must partition $V(G)$ differently. In general, this is false: per \cref{lemma-same-fiber-aut} below, counterexamples arise precisely when $\aut(H)$ acts nontrivially on $V(H)$.

The second proof given in \cite{amt-95-etds}, constructing $V(M(G))$ by successive refinements of an initial state partition, can be made complete by observing that all of the state maps corresponding to the successive refinements are invariant under $\aut(G)$.

\begin{prop}[Lemma 3.1 in \cite{amt-95-etds}]\label{prop-amt-right-meet}
Let $G, H_1, H_2$ be graphs with $H_i \leq_R G$ via $\Psi_i \in \hom_R(G, H_i)$, with partitions $\alpha_i = V(H_i)$ of $V(G)$. Then there is a graph $K \leq_R H_1, H_2$, with $V(K)$ equal to the finest common coarsening of $\alpha_1, \alpha_2$.
\end{prop}

\begin{cor}[Theorem 3.2 in \cite{amt-95-etds}]\label{cor-amt-mg-exists}
Let $G$ be a graph. Then there exists a unique $\leq_R$-minimal graph $M(G) \leq_R G$, with $V(M(G))$ given by the finest common coarsening of the partitions $\alpha = V(H)$ of $V(G)$, where $H \leq_R G$.
\end{cor}

We will show further that for any $\Phi_1, \Phi_2 \in \hom_R(G,M(G))$, we have $\partial \Phi_1 = \partial \Phi_2$. Note that $\partial \Phi_1, \partial \Phi_2$ at least have the same sets of fibers, since, if they did not, then by \cref{prop-amt-right-meet}, we could take their finest common coarsening, contradicting the minimality of $M(G)$. It is therefore enough to show that $\aut(M(G))$ acts trivially on $M(G)$. For this, we need a lemma.

\begin{lemma}\label{lemma-same-fiber-aut}
Let $G,H$ be graphs with $H \leq_R G$, and let $\Phi_1, \Phi_2 \in \hom_R(G,H)$. Suppose that the $\partial \Phi_i$ have the same fibers, i.e. for any $I_1, I_2 \in V(G)$, we have $\partial \Phi_1(I_1) = \partial \Phi_1(I_2)$ if and only if $\partial \Phi_2(I_1) = \partial \Phi_2(I_2)$. Then there exists an automorphism $\tau \in \aut(H)$ such that $\partial \Phi_2 = \partial(\tau \circ \Phi_1)$. 
\end{lemma}

\begin{proof}
By the assumption of equal fibers, there exists a (unique) bijection $T: V(H) \to V(H)$ with $\partial  \Phi_2 = T \circ \partial \Phi_1$. We need to find $\tau \in \aut(H)$ with $\partial \tau = T$. Let $I \in V(H)$. For $J \in F(I)$, note, by the right-resolving hypothesis on the $\Phi_i$ and the choice of $T$, that $|E_{IJ}(H)| = |E_{T(I)T(J)}(H)|$. Letting $\tau|_{E_{IJ}(H)}: E_{IJ}(H) \to E_{T(I)T(J)}(H)$ be any bijection, we have $\tau \in \aut(H)$ with $T = \partial \tau$ and $\partial \Phi_2 = \partial(\tau \circ \Phi_1)$.
\end{proof}

We now discuss the quotient of a graph by its automorphism group. Let $G$ be a graph. We will construct a graph $K = G/\aut(G) \leq_R G$ as follows. Let $V(K)$ consist of the orbits in $V(G)$ under $\aut(G)$. Let $I,J \in V(K)$ and $I_1, I_2 \in I$. We need to specify $|E_{IJ}(K)|$. We claim that $\sum_{J' \in J} |E_{I_1 J'}(G)| = \sum_{J' \in J} | E_{I_2 J'}(G)|$. Indeed, let $\tau \in \aut(G)$ such that $I_2 = \partial \tau(I_1)$. Then $\partial \tau(F(I_1)) = \partial \tau (F(I_2))$, and $\tau(J) = J$, so for any $J' \in J \cap F(I_1)$, we have $\partial \tau(J') \in J \cap F(I_2)$. Therefore $\sum_{J' \in J} |E_{I_1 J'}(G)|   \leq_R \sum_{J' \in J} | E_{I_2 J'}(G)|$. Replacing $\tau$ with $\tau^{-1}$, we obtain equality. Define $K$ by specifying that $|E_{IJ}(K)| = \sum_{J' \in J} |E_{I_i J'}(G)|$. This edge count ensures that $K \leq_R G$. 

In particular, if $\aut(G)$ acts nontrivially on $V(G)$, then $K \neq G$, so $G$ is not $\leq_R$-minimal. It follows that, for a $\leq_R$-minimal graph $M$, $\aut(M)$ acts trivially on $V(M)$. This observation, together with \cref{lemma-same-fiber-aut}, shows that for any $\Phi_1, \Phi_2 \in \hom_R(G,M(G))$, we have $\partial \Phi_1 = \partial \Phi_2$.


\subsection{Proof of \cref{thm-struct-sync-comp}}\label{subsec-stab-proofs}

We prove \cref{thm-struct-sync-comp} in several stages. First, in Lemma \ref{lemma-stab-comp}, we determine how stability classes intersect with fibers in compositions of right-resolvers.

\begin{lemma}\label{lemma-stab-comp}
Let $G, K, H$ be graphs with $H \leq_R K \leq_R G$. Let $\Psi \in \hom_R(G,K)$, $\Delta \in \hom_R(K,H)$, and let $\Phi = \Delta \circ \Psi$. Then $\sim_{\Psi}$-classes are intersections of $\sim_{\Phi}$-classes with $\partial \Psi$-fibers. That is, for $I_1, I_2 \in V(G)$, if $I_1 \sim_{\Psi} I_2$, then $I_1 \sim_{\Phi} I_2$; conversely, if $I_1 \sim_{\Phi} I_2$ and moreover $\partial \Psi(I_1) = \partial \Psi(I_2)$, then $I_1 \sim_{\Psi} I_2$.
\end{lemma} 

\begin{proof}
First suppose that $I_1 \sim_{\Psi} I_2$. Then in particular $\partial \Psi(I_1) = \partial \Psi (I_2) = \hat{I}$. Let $I = \partial \Delta(\hat{I})$ and $u \in L_I(H)$, and consider the unique $\lambda \in L_{\Hat{I}}(K)$ such that $\Delta(\lambda) = u$. Since $I_1 \sim_{\Psi} I_2$, there exists $\mu \in L_{t(\lambda)}(K)$ such that $I_1 \cdot \lambda \mu = I_2 \cdot \lambda \mu$. Letting $v = \Delta(\mu)$, we have $I_1 \cdot uv = I_2 \cdot uv$, so $I_1 \sim_{\Phi}  I_2$.

For the converse, suppose that $I_1 \sim_{\Phi} I_2$ and $\partial \Psi(I_1) = \partial \Psi(I_2) = \Hat{I}$. Let $I = \Delta(\Hat{I})$, let $\lambda \in L_{\Hat{I}}(K)$, let $u = \Delta(\lambda)$, and let $v \in L_{t(u)}(H)$ be such that $I_1 \cdot uv = I_2 \cdot uv$. Consider the unique $\mu \in L_{t(\lambda)}(K)$ such that $\Delta(\mu) = v$. Then $I_1 \cdot \lambda \mu = I_2 \cdot \lambda \mu$, so indeed $I_1 \sim_{\Psi} I_2$.
\end{proof}

\cref{cor-comp-air-fiber} follows immediately from \cref{lemma-stab-comp}, and together they comprise \cref{thm-struct-sync-comp}(1).

\begin{cor}\label{cor-comp-air-fiber}
If $\Phi = \Delta \circ \Psi$ is a composition of right-resolvers, then $\Psi$ is synchronizing if and only if every $\partial \Psi$-fiber is contained in a $\sim_{\Phi}$-class.
\end{cor}

\begin{proof}[Proof of \cref{thm-struct-sync-comp}(2)]
By assumption, the $\sim_{\Phi}$ classes are precisely the $\partial \Psi$-fibers, so by \cref{lemma-stab-comp}, they are also the $\sim_{\Psi}$ classes. Thus indeed $\Psi \in \hom_S(G,K)$. 

Moreover, let $I \in V(H)$, let $I'_1, I'_2 \in (\partial \Delta)^{-1}(I)$, and suppose that $I'_1 \sim_{\Delta} I'_2$. We claim that $I'_1 = I'_2$. Toward this end, we claim that $(\partial \Psi)^{-1}(\{ I'_1, I'_2 \})$ is a subset of a $\sim_{\Phi}$ class. Indeed, let $u \in L_I(H)$, and let $I_i \in (\partial \Psi)^{-1}(I'_i)$. Let $v \in L_{t(u)}(H)$ be such that $I'_1 \cdot uv = I'_2 \cdot uv = J'$. Let $J_i = I_i \cdot uv$. Then $J_1, J_2 \in (\partial \Psi)^{-1}(J')$. Since $\Psi$ is synchronizing, there exists $\gamma \in L_{J'}(K)$ with $J_1 \cdot \gamma = J_2 \cdot \gamma$. Let $w = \Delta(\gamma)$. Then $I_1 \cdot uvw = I_2 \cdot uvw$. Since $u \in L_I(H)$ was arbitrary, it follows that $I_1 \sim_{\Phi} I_2$ as claimed. Since, by assumption, a $\sim_{\Phi}$ class is precisely a $\partial \Psi$-fiber of a single state of $K$, we must have $I'_1 = I'_2$ as claimed. Therefore $\sim_{\Delta}$ is indeed trivial.
\end{proof}

\begin{lemma}\label{lemma-pullback-stab}
Let $G,K,H$ be graphs with $H \leq_R K \leq_R G$. Let $\Psi \in \hom_R(G,K)$, $\Delta \in \hom_R(K,H)$, and let $\Phi = \Delta \circ \Psi$. Let $I \in V(H)$, let $I'_1, I'_1 \in (\partial \Delta)^{-1}(I)$, and let $I_i \in (\partial \Psi)^{-1}(I'_i)$. If $I_1 \sim_{\Phi} I_2$, then $I'_1 \sim_{\Delta} I'_2$.
\end{lemma}

\begin{proof}
Suppose that $I_1 \sim_{\Phi} I_2$. Let $u \in L_I(H)$. Let $v \in L_{t(u)}(H)$ be such that $I_1 \cdot uv = I_2 \cdot uv$. Then, since $I'_i \cdot uv = \partial \Phi(I_i \cdot uv)$, we have $I'_1 \cdot uv = I'_2 \cdot uv$. Therefore $I'_1 \sim_{\Delta} I'_2$.
\end{proof}

\begin{proof}[Proof of \cref{thm-struct-sync-comp}(3)]
Suppose that $\sim_{\Delta}$ is trivial. Let $I \in V(H)$, let $I'_1, I'_2 \in (\partial \Delta)^{-1}(I)$ with $I'_1 \neq I'_2$, and let $I_i \in (\partial \Psi)^{-1}(I'_i)$. Since $\sim_{\Delta}$ is trivial, there exists $u \in L_I(H)$ such that, for every $v \in L_{t(u)}(H)$, we have $I'_1 \cdot uv \neq I'_2 \cdot uv$. This implies that $I_1 \cdot uv \neq I_2 \cdot uv$, so $I_1 \not\sim_{\Phi} I_2$. It follows that each $\sim_{\Phi}$ class is contained inside a $\partial \Psi$ fiber. By \cref{lemma-stab-comp}, this shows that $\sim_{\Phi} = \sim_{\Psi}$. 
\end{proof}

The final part of \cref{thm-struct-sync-comp} can be proved in the strongly connected case using symbolic dynamics, via the multiplicativity of degree under composition of right-resolvers; see \cite{lm-95-intro}, \S 9.1. Using the theory developed so far, we give a self-contained proof without the assumption of strong connectedness.

\begin{proof}[Proof of \cref{thm-struct-sync-comp}(4)]
Suppose that $\Psi, \Delta$ are synchronizing. Let $I \in V(H)$. Let $I_1, I_2 \in (\partial \Phi)^{-1}(I)$. We need to show that $I_1 \sim_{\Phi} I_2$. Let $u \in L_I(K)$. We need to find $v \in L_{t(u)}(H)$ such that $I_1 \cdot uv = I_2 \cdot uv$. Let $I'_i = \partial \Psi(I_i)$. Then $\partial \Delta(I'_1) = \partial \Delta(I'_2)$, so $I'_1 \sim_{\Delta} I'_2$ since $\Delta$ is synchronizing. Let $v_1 \in L_{t(u)}(H)$ be such that $I'_1 \cdot uv_1 = I'_2 \cdot uv_1$. Note that $I'_i \cdot uv_1 = \partial \Psi(I_i \cdot uv_1)$. Let $J' = I'_i \cdot uv_1$ and $J_i = I_i \cdot uv_1$. Then $\partial \Psi(J_1) = J' = \partial \Psi(J_2)$, so $J_1 \sim_{\Psi} J_2$ since $\Psi$ is synchronizing. Let $\gamma \in L_{J'}(K)$ be such that $J_1 \cdot \gamma = J_2 \cdot \gamma$. Let $v_2 = \Delta(\gamma)$. Then $I_i \cdot uv_1 v_2 = J_i \cdot \gamma$, so taking $v = v_1 v_2$, we have $I_1 \cdot uv = I_2 \cdot uv$. Thus $I_1 \sim_{\Phi} I_2$, so $\Phi$ is indeed synchronizing.

We prove the converse by the contrapositive. Suppose that $\Delta$ is not synchronizing. Then there exist $I \in V(H)$ and $I'_1, I'_2 \in (\partial \Delta)^{-1}(I)$ such that $I'_1 \not\sim_{\Delta} I'_2$. By \cref{lemma-pullback-stab}, there exist $I_i \in (\partial \Psi)^{-1}(I'_i)$ such that $I_1 \not\sim_{\Phi} I_2$. Thus $\Phi$ is not synchronizing. Similarly, suppose that $\Psi$ is not synchronizing. Then there exist $I' \in V(K)$ and $I_1, I_2 \in (\partial \Psi)^{-1}(I')$ such that $I_1 \not\sim_{\Psi} I_2$. By \cref{lemma-stab-comp}, $I_1 \not\sim_{\Phi} I_2$, so $\Phi$ is not synchronizing.
\end{proof}


\subsection{Proofs of Propositions \ref{prop-alm-bunchy-equiv} and \ref{cor-alm-stab-unique}}\label{sec-bunchy-details}

\begin{proof}[Proof of \cref{prop-alm-bunchy-equiv}]
Let $M = M(G)$. First, suppose that $G$ is almost bunchy. Let $\Phi_1, \Phi_2 \in \hom_R(G,M)$. For each $I,J \in V(M)$, and each $I' \in \Sigma_G^{-1}(I)$, let $A_{I',J} = \{ e \in E_{I'}(G) \, | \, \Sigma_G(t(e)) = J \}$. If there exists $I' \in \Sigma_G^{-1}(I)$ with $|F(I') \cap \Sigma_G^{-1}(J)| \geq 2$ (by almost bunchiness, there is at most one such $I'$ for any given $I,J$), then let $\tau_{IJ} \in P(M)$ be the permutation of parallel edges in $M$ given by
\begin{align*}
    \tau_{IJ}|_{E(M) \sm E_{IJ}(M)} &= \id|_{E(M) \sm E_{IJ}(M)} \\
    \tau_{IJ}|_{E_{IJ}(M)} &= \Phi_1|_{A_{I',J}} \circ (\Phi_2|_{A_{I',J}})^{-1}
\end{align*}
If there is no $I' \in \Sigma_G^{-1}(I)$ with $|F(I') \cap \Sigma_G^{-1}(J)| \geq 2$, then let $\tau_{IJ}=\id$. Distinct $\tau_{IJ}$ are permutations of disjoint sets and therefore commute. Let $\tau = \prod_{I,J \in V(M)} \tau_{IJ}$. Note that if $G$ is bunchy, then $\tau_{IJ} = \id$ for all $I,J$, so $\tau = \id$.

Now, for each $I,J \in V(M)$ and each $I' \in \Sigma_G^{-1}(I)$, let $\sigma_{IJ,I'} \in P(G)$ be given by
\begin{align*}
    \sigma_{IJ,I'}|_{E(G) \sm A_{I',J}} &= \id|_{E(G) \sm A_{I',J}} \\
    \sigma_{IJ,I'}|_{A_{I',J}} &=  \Phi_2|_{A_{I',J}}^{-1} \circ \tau^{-1} \circ \Phi_1|_{A_{I',J}} 
\end{align*}
All of the $\sigma_{IJ,I'}$ commute. Let $\sigma = \prod_{I,J \in V(M), \, I' \in \Sigma_G^{-1}(I)} \sigma_{IJ,I'}$. Then for $I \in V(M)$, $J \in F(I)$, and $I' \in \Sigma_G^{-1}(I)$,
\[
\tau \circ \Phi_2 \circ \sigma|_{A_{I',J}} = \tau \circ \Phi_2 \circ (\Phi_2|_{A_{I',J}}^{-1} \circ \tau^{-1} \circ \Phi_1|_{A_{I',J}} ) = \Phi_1|_{A_{I',J}}
\]
This concludes the proof in the ``only if'' direction.

For the ``if'' direction, which we prove in the contrapositive, suppose that $G$ is not almost bunchy. Let $I,J \in V(M)$ and $I_1, I_2 \in \Sigma_G^{-1}(I)$ such that $|F(I_i) \cap \Sigma_G^{-1}(J)| \geq 2$ for $i=1,2$. Let $e_{i,1}, e_{i,2} \in A_{I_i,J}$ be such that $t(e_{i,1}) \neq t(e_{i,2})$. Let $a_1, a_2 \in E_{IJ}(M)$, and let $\Phi_1, \Phi_2 \in \hom_R(G,M)$ be such that $\Phi_1(e_{1,j}) = a_j$ but $\Phi_2(e_{2,1}) = a_2$ and $\Phi_2(e_{2,2}) = a_1$. (The behaviour of $\Phi_i$ on $E(G) \sm \{ e_{i,j} \}_{i,j=1,2}$ is irrelevant.) Then there do not exist $\sigma \in P(G)$, $\tau \in P(M)$ such that $\Phi_1 = \tau \circ \Phi_2 \circ \sigma$.

Finally, suppose that $G$ is almost bunchy, but not bunchy. Let $I,J \in V(M)$ and $I' \in \Sigma_G^{-1}(I)$ such that $|F(I') \cap \Sigma_G^{-1}(J)| \geq 2$. Let $e_1, e_2 \in A_{I', J}$ be such that $t(e_1) \neq t(e_2)$. Let $a_1, a_2 \in E_{IJ}(M)$, and let $\Phi_1, \Phi_2 \in \hom_R(G,M)$ be such that $\Phi_1(e_i) = a_i$, $\Phi_2(e_1) = a_2$, and $\Phi_2(e_2) = a_1$. Then there does not exist $\sigma \in P(G)$ such that $\Phi_1 = \Phi_2 \circ \sigma$.
\end{proof}

\begin{proof}[Proof of \cref{cor-alm-stab-unique}]
Let $M = M(G)$. It is enough to show that $S_{\Phi_1}, S_{\Phi_2}$ share a generating set $T$, which we now construct and examine. For $I,J \in V(M)$, if $J \notin F(I)$, let $T_{I,J,1} = T_{I,J,2} = \emptyset$. If $J \in F(I)$, then  let $T_{I,J,i}$ be the set of maps $f_{a,i}: \Sigma_G^{-1}(I) \to \Sigma_G^{-1}(J)$ of the form $I' \mapsto t \circ  (\Phi_i|_{E_{I'}})^{-1}(a)$, i.e. $I' \cdot f_{a,i} = I' \cdot a$ with respect to $\Phi_i$, where $a \in E_{IJ}(M)$. Clearly $T_{I,J,i}$ generates $S_{\Phi_i}$, in the sense that $S_{\Phi_i}$ is the smallest collection of maps closed under composition and containing the $T_{I,J,i}$ as $I,J$ range over $V(M)$.

We claim that $T_{I,J,1} = T_{I,J,2}$. Indeed, let $a \in E_{IJ}(M)$ and let $I' \in \Sigma_G^{-1}(I)$ such that $|F(I') \cap \Sigma_G^{-1}(J)| \geq 2$. It is enough to show that $f_{a,1} \in T_{I,J,2}$, as this will show that $T_{I,J,1} \se T_{I,J,2}$, from which equality follows by symmetry. By almost bunchiness, there is at most one state $I' \in \Sigma_G^{-1}(I)$ with $|F(I') \cap \Sigma_G^{-1}(J)| \geq 2$. If there is no such state, then clearly $f_{a,1} = f_{a,2}$, so assume that such a state $I'$ exists. Observe that $I' \cdot f_{b, 2} = I' \cdot f_{a,1}$ where $b = \Phi_2 \circ (\Phi_1|_{E_{I'(G)}})^{-1} (a)$. Moreover, for every $I'' \in \Sigma_G^{-1}(I)$ with $I'' \neq I'$, we also have, by almost bunchiness, that $I'' \cdot f_{b,2} = I'' \cdot f_{a,1}$. Therefore $f_{a,1} = f_{b,2} \in T_{I,J,2}$, so indeed $T_{I,J,1} = T_{I,J,2}$ as claimed.  Let $T_{I,J} = T_{I,J,i}$ and let $T = \bigcup_{I,J} T_{I,J}$. Then $T$ generates both $S_{\Phi_1}, S_{\Phi_2}$, so indeed $S_{\Phi_1} = S_{\Phi_2}$.
\end{proof}

\subsection{Proof of \cref{thm-univ-prop}}\label{subsec-univ-prop-proof}

\begin{proof}[Proof of \cref{thm-univ-prop}]
Define $T: V(G) \to V(P)$ as follows: for $I' \in V(G)$, let $T(I') = (\partial \Phi_1(I'), \partial \Phi_2(I'))$. Note that $\partial \Psi_P \circ T = \partial(\Psi_i \circ \Phi_i) = \Sigma_G$ for each $i$. Let $C$ be the subgraph of $P$ induced by $T(V(G))$. Let $I' \in V(G)$ and $J' \in F(I')$; $I = \Sigma_G(I')$ and $J = \Sigma_G(J')$; and $I_i = \partial \Phi_i(I')$ and $J_i = \partial \Phi_i(J')$. Observe that $(I_1, I_2) = T(I')$ and $(J_1, J_2) = T(J')$. 

As in the proof of \cref{prop-alm-bunchy-equiv}, let $A_{I',J} = \{ e \in E_{I'}(G) \, | \, \Sigma_G(t(e)) = J  \}$. Since $H_i$ is bunchy, we have $F(I_i) \cap \Sigma_{H_i}^{-1}(J) = \{ J_i \}$, so $\Phi_i(A_{I', J}) = E_{I_i J_i}(H_i)$. It therefore makes sense to define
\[
\Delta_i|_{A_{I',J}} = (\Hat{\Psi}_i|_{E_{(I_1,I_2)(J_1, J_2)}(P)})^{-1} \circ \Phi_i|_{A_{I', J}}: A_{I',J} \to E_{(I_1, I_2)(J_1, J_2)}(P)
\]
Gluing these together, we obtain maps $\Delta_i: E(G) \to E(P)$. For $e \in A_{I',J}$, we have $s(\Delta_i(e)) = T(s(e))$ and $t(\Delta_i(e)) = T(t(e))$, so the $\Delta_i: G \to P$ are graph homomorphisms with $\partial \Delta_i = T$.

We now claim that $\Delta_i(E(G)) = E(C)$. Enumerate $E_{IJ}(M) = \{  a^{(1)}, \dots, a^{(k)} \}$, and let $e_i^{(j)} = (\Psi_i|_{E_{I_i J_i}(H_i)})^{-1}(a^{(j)})$. Then
\[
(e_1^{(j)}, e_2^{(j)}) \in E_{(I_1, I_2)(J_1, J_2)}(P) = E_{(I_1, I_2)(J_1, J_2)}(C)
\]
where equality holds because $(I_1, I_2), (J_1, J_2) \in V(C)$ and $C$ is an induced subgraph of $P$. Fixing $(I_1,I_2)$ and varying $J'$, thus varying $(J_1, J_2)$, the sets $E_{(I_1, I_2)(J_1, J_2)}(P)$ exhaust $E_{(I_1, I_2)}(P)$, so $E_{(I_1, I_2)}(C) = E_{(I_1, I_2)}(P)$. Since $\Delta_i|_{E_{I'}(G)}: E_{I'}(G) \to E_{(I_1, I_2)}(P)$ is surjective, we have $E_{(I_1, I_2)}(C) = \Delta_i(E_{I'}(G))$. Thus indeed $E(C) = \Delta_i(E(G))$.

Since we already know that $E_{(I_1, I_2)}(C) = E_{(I_1, I_2)}(P)$ for each $(I_1, I_2) \in V(C)$, it follows that $C$ is indeed a principal subgraph of $P$. Finally, we have
\[
\Hat{\Psi}_i \circ \Delta_i|_{A_{I',J}} = \Hat{\Psi}_i \circ  (\Hat{\Psi}_i|_{E_{(I_1, I_2)(J_1, J_2)}(P)})^{-1} \circ \Phi_i|_{A_{I', J}} = \Phi_i|_{A_{I', J}}
\]
as claimed. This shows that $\Hat{\Psi}_i|_C: C \to H_i$ is surjective, so $H_i \leq_R C$.
\end{proof}


\subsection{Proofs of \cref{lemma-fiber-bunchy} and \cref{prop-bg-exists}, and construction of $B(G)$}\label{subsec-fiber-max-bunchy}

\begin{proof}[Proof of \cref{lemma-fiber-bunchy}]
Let $(I_1, I_2) \in V(P)$ with $I_i \in V(H_i)$. Let $e = (e_1, e_2), e' = (e_1', e_2') \in E_{(I_1, I_2)}(P)$ such that $t( \partial \Psi_P (e)) = t( \partial \Psi_P (e'))$. We must show that $t(e) = t(e')$. Toward that goal, note that for each $i =1,2$, we have
\begin{align*}
    t(\partial \Psi_P(e)) &= \partial \Psi_P(t(e)) = \partial \Psi_i(t(e_i)) \\
    t(\partial \Psi_P(e')) &= \partial \Psi_P(t(e')) = \partial \Psi_i(t(e'_i))
\end{align*}
Therefore $\partial \Psi_i(t(e_i)) = \partial \Psi_i(t(e'_i))$ for each $i$, since $t( \partial \Psi_P (e)) = t( \partial \Psi_P (e'))$ by hypothesis. But since the $H_i$ are bunchy, we in fact have $t(e_i) = t(e_i')$ for each $i$, so indeed $t(e) = t(e')$ as required. This shows that $\partial \Psi_P|_{F((I_1, I_2))}$ is a bijection, so $P$ is bunchy.

Let $C$ be a principal subgraph of $P$ such that $\partial \Hat{\Psi}_i|_{V(C)} : V(C) \to V(H_i)$ are surjective. Then $M(C) = M$. For each $J \in V(M)$, we have $\Sigma_C^{-1}(J) = \Sigma_P^{-1}(J) \cap V(C)$. Thus $C$ is bunchy.
\end{proof}

\begin{proof}[Proof of \cref{prop-bg-exists}]
To prove (1), let $H_1, H_2 \leq_R G$ be $\leq_R$-maximal among the bunchy right-resolving factors of $G$. Let $\Phi_i \in \hom_R(G,H_i)$ and $\Psi_i \in \hom_R(H_i, M)$. Let $P = H_1 \times_{\Psi_1, \Psi_2} H_2$. By \cref{thm-univ-prop}, there exist a principal subgraph $C$ of $P$, and $\Delta_i \in \hom_R(G,C)$, such that $\Phi_i = \Hat{\Psi}_i|_C \circ \Delta_i$, so $H_i \leq_R C$. By \cref{lemma-fiber-bunchy}, $C$ is bunchy, so by the maximality of the $H_i$, we have $H_1 = C = H_2$. This proves uniqueness, so we can take $B = H_i = C$.

To prove (2), let $H \leq_R G$ be bunchy and let $\Phi \in \hom_R(G,H)$. Let $\Phi' \in \hom(G,B)$, $\Psi \in \hom_R(H,M)$, $\Psi' \in \hom(B,M)$. Let $P = H \times_{\Phi, \Phi'} B$. By the universal property (\cref{thm-univ-prop}), there exist a principal subgraph $C$ of $P$ and $\Delta, \Delta' \in \hom_R(G,C)$ with $\Phi = \Hat{\Psi} \circ \Delta$ and $\Phi' = \Hat{\Psi}' \circ \Delta'$. Again by \cref{lemma-fiber-bunchy}, $C$ is bunchy, so by (1) and the fact that $B \leq_R C$, we have $B = C$. This proves (2) with $\Theta = \Hat{\Psi}'|_C$.
\end{proof}

We now present the construction of $B(G)$ described in Algorithm \ref{alg-construct-bg}.

\begin{prop}\label{prop-construct-bg}
Let $G$ be a graph. For $J_1, J_2 \in V(G)$ with $\Sigma_G(J_1) = \Sigma_G(J_2)$, write $J_1 \approx_0 J_2$ if there exist paths $\gamma = \gamma_1 \cdots \gamma_n, \delta = \delta_1 \cdots \delta_n \in L(G)$ where $\gamma_i, \delta_i \in E(G)$, such that $s(\gamma_1) = s(\delta_1)$, $t(\gamma_n) = J_1$, $t(\delta_n) = J_2$, and $\Sigma_G(t(\gamma_i)) = \Sigma_G(t(\delta_i))$ for each $i$. Let $\approx$ denote the transitive closure of $\approx_0$ and let $\Phi \in \hom_R(G,M(G))$. Then $\approx$ is a congruence with respect to $\Phi$, and $B(G) = G/\approx$.
\end{prop}

\begin{proof}
Let $M = M(G)$. Let $I \in V(M)$, $I_1, I_2 \in \Sigma_G^{-1}(I)$, and $J \in F(I)$. Suppose that $I_1 \approx I_2$. Let $J_i \in F(I_i) \cap \Sigma_G^{-1}(J)$. We need to show that $J_1 \approx J_2$. Let $I_1 = I^{(0)}, I^{(1)}, \dots, I^{(n)} = I_2 \in \Sigma_G^{-1}(I)$ with $I^{(j)} \approx_0 I^{(j+1)}$. Let $\gamma^{(j)}, \delta^{(j)} \in L(G)$ witness the relation $I^{(j)} \approx_0 I^{(j+1)}$. Choose $J^{(j)} \in F(I^{(j)}) \cap \Sigma_G^{-1}(J)$, with $J^{(0)} = J_1$ and $J^{(n)} = J_2$. Let $e^{(j)} \in E_{I^{(j)} J^{(j)}}(G)$. Then $\gamma^{(j)} e^{(j)}, \delta^{(j)} e^{(j+1)}$ witness $J^{(j)} \approx_0 J^{(j+1)}$. This shows that $J_1 \approx J_2$.

Let $\Phi \in \hom_R(G,M(G))$. As in the previous paragraph, let $I \in V(M)$, let $J \in F(I)$, and let $I_1, I_2 \in \Sigma_G^{-1}(I)$. Suppose that $I_1 \approx I_2$. Let $a \in E_{IJ}(M)$, let $J_i = I_i \cdot a$, and let $e_i = (\Phi|_{E_{I_i}(G)} )^{-1}(a)$. Then $t(e_i) = J_i$, so $J_1 \approx J_2$. Therefore $\approx$ is indeed a congruence for $\Phi$. It follows that $F([I_i]_{\approx}) \cap \Sigma_{G/\approx}^{-1}(J) = \{ [J_i]_{\approx} \}$, so indeed $G/\approx$ is bunchy. 

To see that $G/\approx$ is $\leq_R$-maximal among the bunchy factors of $G$, let $H \leq_R G$ be bunchy and $\Psi \in \hom_R(G,H)$. If $I_1 \approx_0 I_2$ and $H$ is bunchy, then we must have $\partial \Psi(I_1) = \partial \Psi(I_2)$. Therefore the partition into $\partial \Psi$-fibers corresponds to an equivalence relation that coarsens the symmetric, reflexive relation $\approx_0$, and thus also coarsens the transitive closure $\approx$. Considering $V(H)$ as a partition of $V(G)$, we must therefore have $V(H) \preceq V(G)/\approx$, so $(G/\approx) \, = B(G)$ by the maximality of $B(G)$.
\end{proof}

\subsection{Proofs of Propositions \ref{prop-bfc-equiv-new} and \ref{cor-air-fact-ext-mg}}\label{subsec-fiber-fact-ext}

\begin{lemma}\label{lemma-princ-fiber-embed}
Let $G_1, G_2, H, K$ be graphs with $K \leq_R H \leq_R G_i$. Let $\Delta \in \hom_R(H,K)$. Let $\Psi_i \in \hom_R(G_i, H)$ and $P = G_1 \times_{\Psi_1, \Psi_2} G_2$. Let $\Phi_i = \Delta \circ \Psi_i$ and $Q = G_1 \times_{\Phi_1, \Phi_2} G_2$. Then, noting that $V(P) = V(Q)$ and $E(P) \se E(Q)$, we have $\Hat{\Phi}_i|_P = \Hat{\Psi}_i$, and every principal subgraph of $P$ is a principal subgraph of $Q$.
\end{lemma} 

\begin{proof}
Let $C$ be a principal subgraph of $P$. By the definition of the fiber product, we have $V(P) = V(Q) = V(G_1) \times V(G_2)$. Let $I \in V(H)$ and let $I_i \in (\partial \Psi_i)^{-1}(I)$. Suppose that $(I_1, I_2) \in V(C)$. We need to show that $E_{(I_1, I_2)}(C) = E_{(I_1, I_2)}(Q)$.

Since $E_{(I_1, I_2)}(C) = E_{(I_1, I_2)}(P)$ by the definition of a principal subgraph, it is enough to show that $E_{(I_1, I_2)}(P) = E_{(I_1, I_2)}(Q)$ for any $(I_1, I_2) \in V(P)$. Clearly $E_{(I_1, I_2)}(P) \se E_{(I_1, I_2)}(Q)$, so it is enough to show that $|E_{(I_1, I_2)}(P)| = |E_{(I_1, I_2)}(Q)|$. To see this equality, note that since $\Delta \circ \Psi_P = \Phi_Q$, we have
\[
|E_{(I_1, I_2)}(P)| = |E_{\partial \Psi_P(I_1, I_2)}(H)| = |E_{\partial \Phi_Q(I_1,I_2)}(K)| = |E_{(I_1,I_2)}(Q)|
\]
where the equalities follow from the facts that $\Psi_P$, $\Delta$, and $\Phi_Q$ respectively are right-resolving. Therefore $E_{(I_1, I_2)}(P) = E_{(I_1, I_2)}(Q)$. This shows that $C$ is indeed a principal subgraph of $Q$. Moreover, for $e = (e_1, e_2) \in E(P)$, we have $\Hat{\Phi}_i(e) = e_i = \Hat{\Psi}_i(e)$, so indeed $\Hat{\Phi}_i|_P = \Hat{\Psi}_i$. 
\end{proof}

\begin{lemma}\label{prop-air-fact-ext}
Let $G_1,G_2, K$ be graphs with $K \leq_S G_i$ via $\Phi_i \in \hom_S(G_i,K)$. Let $C$ be a principal subgraph of $P = G_1 \times_{\Phi_1, \Phi_2} G_2$ such that the $\partial \Hat{\Phi}_i|_{V(C)}: V(C) \to V(G_i)$ are surjective. Then $\Hat{\Phi}_i|_C \in \hom_S(C,G_i)$. In particular, $C$ is a common synchronizing extension of the $G_i$.
\end{lemma}

\begin{proof}
Let $I \in V(K)$ and let $I_i, I'_i \in V(G_i)$ with $\partial \Phi_i(I_i) = \partial \Phi_i(I'_i) = I$. Since the $\Phi_i$ are synchronizing, we have $I_i \sim_{\Phi_i} I'_i$. Suppose that $(I_1, I_2), (I'_1, I'_2) \in V(C)$. In order to show that $\Hat{\Phi}_i|_C \in \hom_S(C,G_i)$, we claim that $(I_1, I_2) \sim_{\Phi_P} (I'_1, I'_2)$. This will show that $\Phi_P|_C \in \hom_S(C,K)$. Since $\Phi_P|_C = \Phi_i \circ \Hat{\Phi}_i|_C$, it will then follow that $\Hat{\Phi}_i|_C \in \hom_S(C,G_i)$ by \cref{thm-struct-sync-comp}(4).

To prove this claim, let $u \in L_I(K)$. Since $I_1 \sim_{\Phi_1} I'_1$, there exists $v_1 \in L_{t(u)}(K)$ such that $I_1 \cdot uv_1 = I'_1 \cdot uv_1$. Similarly, since $I_2 \sim_{\Phi_2} I'_2$, there exists $v_2 \in L_{t(v_1)}(K)$ such that $I_2 \cdot uv_1 v_2 = I'_2 \cdot uv_1 v_2$. Then in particular $(I_1, I_2) \cdot u v_1 v_2 = (I'_1, I'_2) \cdot u v_1 v_2$. Since $u \in L_I(K)$ was arbitrary, we have $(I_1, I_2) \sim_{\Phi_P} (I'_1, I'_2)$ as claimed.
\end{proof}

\begin{proof}[Proof of \cref{cor-air-fact-ext-mg}(1)]
Let $K \leq_S G_i$ and $\Psi_i \in \hom_S(G_i,K)$. Let $P = G_1 \times_{\Psi_1, \Psi_2} G_2$ and let $C$ be a principal subgraph of $P$ such that $\partial \Hat{\Psi}_i|_{V(C)}: V(C) \to V(G_i)$ are surjective. By \cref{prop-air-fact-ext}, we have $\Hat{\Psi}_i|_C \in \hom_S(C,G_i)$. Let $\Delta \in \hom_R(K,M)$, let $\Phi_i = \Delta \circ \Psi_i$, and let $Q = G_1 \times_{\Phi_1, \Phi_2} G_2$. Then, by \cref{lemma-princ-fiber-embed}, $C$ is a principal subgraph of $Q$, with $\Hat{\Phi}_i|_C = \Hat{\Psi}_i|_C$. In particular, $\Hat{\Phi}_i|_C \in \hom_S(C,G_i)$.
\end{proof}

We now give an equivalent form of the $O(G)$ conjecture. Fragments of this result appear in \cref{cor-air-fact-ext-mg}. The logical structure of \cref{prop-og-equiv} (one statement is equivalent to the equivalence of four other statements) is unusual.

\begin{prop}\label{prop-og-equiv}
Let $\cF$ be a family of graphs satisfying the following conditions:
\begin{enumerate}[label=(\roman*)]
    \item If $G \in \cF$ and $H \leq_R G$, then $H \in \cF$.
    
    \item Let $G_1, G_2, K \in \cF$ with $K \leq_R G_i$. Let $\Phi_i \in \hom_R(G_i, K)$, let $P = G_1 \times_{\Phi_1, \Phi_2} G_2$, and let $C$ be a principal subgraph of $P$ such that the $\partial \Hat{\Phi}_i|_{V(C)}$ are surjective. Then $C \in \cF$.
\end{enumerate}

Then the following assertions are equivalent.

\begin{enumerate}[label=(\arabic*)]
    \item For any $G \in \cF$, there exists a unique $\leq_S$-minimal graph $O(G) \leq_S G$.
    
    \item For any $G_1, G_2 \in \cF$, the following assertions are equivalent.
    
    \begin{enumerate}[label=(\alph*)]
        \item $O(G_1), O(G_2)$ exist and are equal.
        
        \item $G_1, G_2$ have a common synchronizing factor.
        
        \item $M(G_1) = M(G_2) = M$ and there exist $\Phi_i \in \hom_R(G_i,M)$ such that $\Hat{\Phi}_i \in \hom_S(C,G_i)$ for some principal subgraph $C$ of $Q = G_1 \times_{\Phi_1, \Phi_2} G_2$.
        
        \item $G_1, G_2$ have a common synchronizing extension $K \in \cF$.

    \end{enumerate}
\end{enumerate}

\end{prop}

\begin{rmk}
The proof of \cref{prop-og-equiv} in fact shows that both (1) and (2) are equivalent to the assertion that (d) implies (b). Note that for $G_1 = G_2 = G$, (a) states that $O(G)$ is well-defined, while (b)--(d) are trivial.
\end{rmk}

\begin{rmk}
The $O(G)$ conjecture states that the equivalent statements (1) and (2) in \cref{prop-og-equiv} hold with $\cF$ equal to the class of all strongly connected graphs.
\end{rmk}

\begin{proof}[Proof of \cref{prop-og-equiv}]
First, assume (2). Let $G$ be a graph and let $H_1, H_2 \leq_S G$ be $\leq_S$-minimal. Since $H_1, H_2$ have the common synchronizing extension $G$, they satisfy condition (d), so they also satisfy condition (b), i.e. there exists a common synchronizing factor $K \leq_S H_i$. But since the $H_i$ were assumed minimal, we must have $H_1 = K = H_2 = O(G)$.

Now, assume (1) and deduce (2) as follows. Trivially, (a) implies (b) and (c) implies (d). Moreover, (b) implies (c) by \cref{cor-air-fact-ext-mg}(1). Finally, assume (d). Suppose that $G_1, G_2$ have a common synchronizing extension $K$. Then $O(G_1) = O(K) = O(G_2)$, so (d) implies (a).
\end{proof}

\begin{proof}[Proof of \cref{cor-air-fact-ext-mg}(2)]
This is immediate from \cref{prop-og-equiv}, specifically, the equivalence of 2(b) and 2(c), with $\cF$ taken to be the class of all strongly connected graphs.
\end{proof}

\begin{proof}[Proof of \cref{cor-air-fact-ext-mg}(3)]
By \cref{thm-alm-og-exists}, this is immediate from \cref{prop-og-equiv}, again via the equivalence of 2(b) and 2(c), but with $\cF$ taken to be the class of strongly connected almost bunchy graphs.
\end{proof}

\begin{prop}\label{prop-bfc-equiv}
Let $\cF$ be a family of graphs such that, if $G \in \cF$ and $H \leq_R G$, then $H \in \cF$. Then the following assertions are equivalent.

\begin{enumerate}[label=(\arabic*)]
    \item Any $\leq_S$-minimal graph $H \in \cF$ is bunchy.

    \item For any $G \in \cF$, there exists some bunchy $H \leq_S G$.
    
    \item For any non-bunchy $G \in \cF$, there exists some $\Phi \in \hom_R(G,M(G))$ with $\sim_{\Phi}$ nontrivial.
    
    \item For any $G \in \cF$, $B(G) \leq_S G$.
\end{enumerate}

\end{prop}

\begin{rmk}
Note that \cref{prop-bfc-equiv} is a more detailed version of \cref{prop-bfc-equiv-new}.
\end{rmk}

\begin{proof}
To see that (1) implies (2), let $G \in \cF$ and consider the set $\{ H \, | \, H \leq_S G  \}$. Being a finite partially ordered set, this set must have at least one minimal element. If (1) holds, then this minimal element is bunchy. Thus (2) holds.

To see that (2) implies (3), let $G \in \cF$ be non-bunchy. There exists at least one $\leq_S$-minimal graph $H \leq_S G$. If (2) holds, then $G$ is not $\leq_S$-minimal, so $H \neq G$. Let $\Psi \in \hom_S(G,H)$. The $\partial \Psi$-fibers are precisely the $\sim_{\Psi}$-classes. Since $H \neq G$, the $\partial \Psi$-fibers are not merely singletons, so $\sim_{\Psi}$ is nontrivial. Let $M=M(G)$ and $\Delta \in \hom_R(H,M)$, and let $\Phi = \Delta \circ \Psi$. By \cref{lemma-stab-comp}, the $\sim_{\Phi}$-classes are unions of $\sim_{\Psi}$-classes, so in particular, $\sim_{\Phi}$ is nontrivial.

To see that (3) implies (1), let $G \in \cF$ be non-bunchy. If (3) holds, then there exists $\Phi \in \hom_R(G,M(G))$ with $\sim_{\Phi}$ nontrivial. Then $G/\sim_{\Phi} \leq_S G$ and $G/\sim_{\Phi} \neq G$; in particular, $G$ is not $\leq_S$-minimal. This proves (1) in the contrapositive.

Finally, we show that (2) and (4) are equivalent. If (4) holds, then (2) holds since $B(G)$ is bunchy. Conversely, assume (2) and let $G \in \cF$. Then there exists $H \leq_S G$ bunchy. Let $\Phi \in \hom_S(G,H)$ and let $B = B(G)$. By \cref{prop-bg-exists}(2) and \cref{thm-struct-sync-comp}(4), we have $\Phi = \Delta \circ \Theta$ for some $\Theta \in \hom_S(G,B)$, $\Delta \in \hom_S(B,H)$. Therefore (4) holds.
\end{proof}

\begin{cor}\label{cor-bunchy-implies-o}
Let $\cF$ be a family of graphs such that, if $G \in \cF$ and $H \leq_R G$, then $H \in \cF$. Suppose that any $\leq_S$-minimal element of $\cF$ is bunchy. Then $O(G)$ exists for any $G \in \cF$.
\end{cor}

The proof of \cref{cor-bunchy-implies-o} is a trivial adaptation of the proof of \cref{prop-bunchy-implies-o}.

\begin{proof}[Proof of \cref{cor-air-fact-ext-mg}(4)]
Assuming the bunchy factor conjecture, the $O(G_i)$ are well-defined. By the equivalence of conditions (a),(c) in \cref{prop-bfc-equiv}, the hypothesis on the $B(G_i)$ is equivalent to the equality $O(B(G_1)) = O(B(G_2))$, which in turn is equivalent to the equality $O(G_1) = O(G_2)$ by \cref{prop-bfc-equiv}.
\end{proof}


\section{Proof of \cref{thm-bfc-cyc-bunch}, following Trahtman}\label{sec-traht}

In this section, we recall Trahtman's proof of the road colouring theorem and reformulate it in a form applicable to \cref{thm-bfc-cyc-bunch}.

\subsection{Systems of maps with unique tallest trees}\label{subsec-traht-utt}

Trahtman's proof uses an idea well known in the literature on the combinatorics of transformations of finite sets---namely, that, the graph of a transformation $f$ of a set $V$ is a directed graph of constant out-degree $1$ with $V$ for its set of states where, for each state $I \in V$, the unique edge with source $I$ has target $I \cdot f$. When $V$ is finite, the graph consists of a set of state-disjoint directed cycles together with trees rooted on the cycles, directed toward their roots. Trahtman's proof is based on the construction of a transformation with a unique tallest tree (the height of a tree being the maximum distance from the root to a leaf). We need a few definitions in order to extend the notion of a unique tallest tree to a system of maps taking one set to another, cyclically, as opposed to a transformation of a single set.

Let $p \in \N$ and let $\{ V_i  \}_{0 \leq i \leq p-1}$ be disjoint finite sets. Let $a_i: V_i \to V_{i+1}$ be maps, with subscripts read modulo $p$. For $k \geq 0$, let $b_k = a_k a_{k+1} \cdots a_{p-1} a_0 \cdots a_{k-1}: V_k \to V_k$. As noted above in the general discussion of graphs of transformations, since each $b_k$ is a transformation of a finite set, it is eventually periodic---that is, for each $k$, there exist $m \geq 0$, $z \geq 1$ such that $b_k^{m+z} = b_k^m$. Moreover, the orbits of individual elements $I \in V_k$ may vary in their eventually periodic behaviour. Specifically, for a given $I \in V_k$, consider the  lexicographically minimal $(\ell,m,z)$ with $0 \leq \ell \leq p-1$, $m \geq 0$, $z \geq 1$, such that $I \cdot a_k \cdots a_{k+\ell-1} b_{k+\ell}^{m+z} = I \cdot a_k \cdots a_{k+\ell-1} b_{k+\ell}^{m}$. Define the \textit{height} $h(I) = mp+\ell$ and the \textit{root} $\rho(I) = I \cdot a_k \cdots a_{k+\ell-1} b_{k+\ell}^{m+z}$. The idea is that $h(I)$ is the number of steps required until the orbit of $I$ reaches the root $\rho(I)$ and becomes periodic, with $z(I)=z$ being the period as a multiple of $p$.

To talk about unique tallest trees, let $h_{\max,k} = \max \{ h(I) \, | \, I \in V_k  \}$, and let $h_k(J) = \max \{ h(I) \, | \, I \in V_k, \, \rho(I) = J \}$. Let $z_k = \mathrm{lcm} \{ z(I) \, | \, I \in V_k \}$. We say that the system $(V_i, a_i)_{0 \leq i \leq p-1}$ has a \textit{unique tallest tree} at $V_k$ if there is a unique $J$ with $h_k(J) = h_{\max,k}$. Note that the terms we have defined here still make sense even if the $V_i$ are not all pairwise disjoint, as long as any two are either equal or disjoint, since we can make them disjoint by replacing $V_i$ with $V_i \times \{ i \}$.

We now present our interpretation of a key step in Trahtman's proof, applying \cref{lemma-min-diff-stab} and closely following \cite{bpr-09-ca}. 

\begin{lemma}\label{lemma-utt-stab}
Let $G,H$ be strongly connected graphs with $H \leq_R G$. Let $\Phi \in \hom_R(G,H)$. Let $I_0, \dots I_{p-1} \in V(H)$, not necessarily distinct, be such that $I_{i+1} \in F(I_i) \neq \emptyset$, and let $a_i \in E_{I_i I_{i+1}}(H)$ (with subscripts read modulo $p$). Suppose that the system $((\partial \Phi)^{-1}(I_i), a_i)_{0 \leq i \leq p-1}$ has a unique tallest tree, where we write $a_i$ for the map $I \mapsto I \cdot a_i$, $I \in (\partial \Phi)^{-1}(I_i)$. Then $\sim_{\Phi}$ is nontrivial.
\end{lemma}

\begin{proof}
For $0 \leq i \leq p-1$, let $V_i = (\partial \Phi)^{-1}(I_i)$. Suppose without loss of generality that the system has a unique tallest tree at $V_0$. Let $I \in V_0$ be a state of maximal height $h(I) = h_{\max,0}$. Let $m \geq 0$, $0 \leq p-1$ be such that $h_{\max,0} = mp + \ell$, and let $z = z(I)$ and let $R = \rho(I)$. Note that $R \in V_{\ell}$. That is, upon cyclic application of the maps $a_i$, $I$ is eventually mapped to $R$, then returns to $R$ every $z$ cycles around the graph.

By strong connectedness and \cref{lemma-sc-min-img}, let $U \se I_0$ be a minimal image such that $I \in U$. We claim that there is no other state $I' \in U$ with $\rho(I') = R$ and $I' \neq I$. Indeed, suppose that there is such a state $I'$. Then $I \cdot a_0 \cdots a_{\ell-1} b_{\ell}^m = I' \cdot a_0 \cdots a_{\ell-1} b_{\ell}^m = R$, so $| U \cdot a_0 \cdots a_{\ell-1} b_{\ell}^m | < |U|$, contradicting the minimality of $U$. This proves the claim. Let $U_0 = U \sm \{ I \}$. Then every element of $U_0$ has height strictly less than $h_{\max,0}$.

Let $u_1 = a_0 \cdots a_{\ell-1} b_{\ell}^{m-1} a_{\ell} \cdots a_{\ell-2}: V_0 \to V_{\ell-1}$, where the tail $a_{\ell} a_{\ell+1} \cdots a_{\ell-2}$ includes each $a_i$ exactly once, other than $a_{\ell-1}$, and the subscripts are read modulo $p$. The effect of applying $u_1$ to $V_0$ is to bring $I$ to one step before its first encounter its root $R$; since $I$ has maximal height, every other $I' \in V_0$ has already reached its root and is in the periodic part of its orbit after application of $u_1$.

Let $u_2 = u_1 b_{\ell-1}^{z_{\ell}} = a_0 \cdots a_{\ell-1} b_{\ell-1}^{m + z_k} a_{\ell} \cdots a_{\ell-2}: V_0 \to V_{\ell-1}$. Let $J_i = I \cdot u_i$ and $U_i = U \cdot u_i$. As observed in the previous paragraph, after application of $u_1$, $I$ is not yet in the periodic part of its orbit (i.e. the orbit under cyclic application of the maps $a_i$), but every other element of $U_1$ is in the periodic part of its orbit. 

Observe that $J_1 \neq J_2$ by the assumed value of $h(I)$. However, since $I$ is the unique element of $U$ with this maximal height, we have $U_1 \Delta U_2 = \{ J_1, J_2 \}$. Since the $U_i$ are minimal images, we have $J_1 \sim_{\Phi} J_2$ by \cref{lemma-min-diff-stab}.
\end{proof}


\subsection{Obtaining a right-resolver with a unique tallest tree}\label{subsec-traht-app}

Let $G$ be a graph. We define a \textit{total order colouring} to be a total ordering of each edge set $E_I(H)$, i.e. a labeling of the edges of $G$ such that, if $|E_{I}(G)| = k$, then the edges in $E_{I}(G)$ are labeled bijectively by $\{ 0, \dots, k-1 \}$. Suppose that $M = M(G)$ is a cycle of bunches. Then, once a total order colouring of $M$ is fixed, total order colourings of $G$ correspond bijectively with right-resolvers $\Phi \in \hom_R(G,M)$. 

Letting $V(M) = \{ I_0, \dots, I_{p-1} \}$, there is exactly one edge $a_i \in E_{I_i}(H)$ labeled $0$ for each $i$. This yields a subgraph $W$ of $G$ consisting of edges labeled $0$, which is a spanning subgraph of $G$ of constant out-degree $1$. Every graph of constant out-degree $1$ consists of a set of state-disjoint cycles, together with trees rooted on the cycles, directed toward their roots. The height of a tree is the maximum path length from a state in the tree to its root. Observe that the system $( (\partial \Phi)^{-1}(I_i), a_i )_{i=0}^{p-1}$ has a unique tallest tree, in the sense of mappings, if and only if there is a unique tallest tree in $W$.

We now present our interpretation of the main technical lemma in the proof of the road colouring theorem (Lemma 10.4.6 in \cite{bpr-09-ca}). The following is not how the lemma is stated in \cite{bpr-09-ca}, but one can follow the proof and observe that it is equivalent. 

\begin{lemma}[Trahtman]\label{lemma-traht-utt}
Let $G$ be a strongly connected graph such that $M(G)$ is a cycle of bunches. At least one of the following is true:
\begin{enumerate}
    \item $G$ is itself a cycle of bunches.

    \item $G$ has two distinct bunches whose outgoing edges have the same target.
    
    \item $G$ admits a total order colouring with a unique tallest tree.
\end{enumerate}
\end{lemma}

With this result, we can prove our generalization of the road colouring theorem.

\begin{proof}[Proof of \cref{thm-bfc-cyc-bunch}]
Let $M = M(G)$. The claim is trivially true if $|V(G)| = |V(M)|$, in which case $G=M$. Suppose that it is true for all $H$ with $|V(H)| \leq_R N$ and $M(H) = M$. Suppose that $|V(G)| = N+1$. If $G$ is bunchy, then we are done. If $G$ is not bunchy, but has two states that can be in-amalgamated, then they are stable for some $\Phi \in \hom_R(G,M)$ by \cref{lemma-amalg-stab}. Let $G' = G/\sim_{\Phi}$. Then $|V(G/\sim_{\Phi})| < |V(G)|$, so by the inductive hypothesis, there is some bunchy $H \leq_S G/\sim_{\Phi} \, \leq_S G$. 

Now, suppose that $G$ does not have two states that can be in-amalgamated---in particular, $G$ does not have two distinct bunches whose outgoing edges have the same target. Since $G$ is not bunchy, it is in particular not a cycle of bunches, so by \cref{lemma-traht-utt}, it admits a total order colouring with unique tallest tree. As remarked above, this total order colouring corresponds to some right-resolver $\Phi \in \hom_R(G,M)$, which has $\sim_{\Phi}$ nontrivial by \cref{lemma-utt-stab}. Then once more $G/\sim_{\Phi}$ is strictly smaller than $G$. We can then apply \cref{prop-bfc-equiv} to conclude that any $\leq_S$-minimal $O$ with $M =M(O)$ is a cycle of bunches. If $O \leq_S G$, then $O=O_{M,p}$ for $p=\per(G)/\per(M)$.
\end{proof}

\section*{Acknowledgments}

The author thanks Brian Marcus and Tom Meyerovitch for their patient, generous advice and supervision; Theo Morrison for a number of insightful questions; the anonymous referees for very careful reading and suggestions; and Mike Boyle for suggesting (to Brian Marcus) a renewed attack on the $O(G)$ problem, shortly after the road problem was solved in 2007. The author was supported through NSERC Discovery Grant EJYR GR010163 (PI: Brian Marcus).

\end{document}